\documentclass[a4paper,12pt,reqno]{amsart}
\usepackage{amsmath, amsfonts, amssymb}
\usepackage[hang,flushmargin]{footmisc} 
\usepackage{amsthm}
\usepackage{hyperref}
\usepackage{etoolbox}

\newtheorem{theorem}{Theorem}[section]
\newtheorem{lemma}[theorem]{Lemma}

\newtheorem{proposition}[theorem]{Proposition}

\theoremstyle{definition}
\newtheorem{defn}[theorem]{Definition}
\newtheorem{remark}[theorem]{Remark}

\newtheoremstyle{break}
  {}
  {}
  {\itshape}
  {}
  {\bfseries}
  {.}
  {\newline}
  {}

\theoremstyle{break}
\newtheorem{theorembk}[theorem]{Theorem}
\newtheorem{propositionbk}[theorem]{Proposition}

\def\E{\mathbb{E}}
\def\Z{\mathbb{Z}}
\def\R{\mathbb{R}}
\def\T{\mathbb{T}}
\def\C{\mathbb{C}}
\def\N{\mathbb{N}}

\def\F{\mathbb{F}}
\def\Q{\mathbb{Q}}
\def\X{\mathcal{X}}

\def\Ghor{G^{\triangledown}}

\newcommand{\id}{\mathrm{id}}

\DeclareMathOperator{\hcf}{hcf}
\DeclareMathOperator{\lcm}{lcm}
\DeclareMathOperator{\poly}{poly}
\DeclareMathOperator{\Span}{Span}
\DeclareMathOperator{\Lip}{Lip}

\newcommand{\nil}{\mathrm{nil}}
\newcommand{\sml}{\mathrm{sml}}
\newcommand{\unf}{\mathrm{unf}}
\newcommand{\threeAP}{\mathrm{3AP}}

\DeclareMathOperator{\cN}{\mathcal{N}}

\newcommand{\Zmod}[1]{\Z_{#1}} 

\let\originalleft\left
\let\originalright\right
\renewcommand{\left}{\mathopen{}\mathclose\bgroup\originalleft}
\renewcommand{\right}{\aftergroup\egroup\originalright}

\providecommand{\abs}[1]{\left\lvert #1 \right\rvert}
\providecommand{\norm}[1]{\left\lVert #1 \right\rVert}

\providecommand{\floor}[1]{\left\lfloor#1\right\rfloor}
\providecommand{\Sol}{S}
\providecommand{\md}{d}
\providecommand{\lf}{\varphi}
\providecommand{\Lf}{\varPhi}

\renewcommand{\subset}{\subseteq}
\renewcommand{\supset}{\supseteq}

\title[Linear forms on cyclic groups and periodic nilsequences]{Convergence results for systems of linear forms on cyclic groups, and periodic nilsequences}
\date{}

\author{Pablo Candela}
\address{D\'epartement de math\'ematiques et applications\\
	\'Ecole normale sup\'erieure\\
	45 rue d'Ulm\\
	F 75230 Paris cedex 05\\
	France}
\email{pablo.candela@ens.fr}

\author{Olof Sisask}
\address{Department of Mathematics\\
	KTH\\
	SE-100 44  Stockholm\\
	Sweden}
\email{sisask@kth.se}


\begin{document}

\begin{abstract}
Given a positive integer $N$ and real number $\alpha\in [0, 1]$, let $m(\alpha,N)$ denote the minimum, over all sets $A\subset \Zmod{N}$ of size at least $\alpha N$, of the normalized count of 3-term arithmetic progressions contained in $A$. A theorem of Croot states that $m(\alpha,N)$ converges as $N\to\infty$ through the primes, answering a question of Green. Using recent advances in higher-order Fourier analysis, we prove an extension of this theorem, showing that the result holds for $k$-term progressions for general $k$ and further for all systems of integer linear forms of finite complexity. We also obtain a similar convergence result for the maximum densities of sets free of solutions to systems of linear equations. These results rely on a regularity method for functions on finite cyclic groups that we frame in terms of periodic nilsequences, using in particular some regularity results of Szegedy (relying on his joint work with Camarena) and equidistribution results of Green and Tao.
\end{abstract}

\maketitle

\section{Introduction}

This paper is concerned with the occurrence of linear configurations in subsets of finite cyclic groups. By a linear configuration in a set $A\subset \Zmod{N}$ we mean a tuple of elements of $A$ that solve a given homogeneous system of linear equations with integer coefficients. Such configurations can also be described as images $(\lf_1(\mathbf{n}),\lf_2(\mathbf{n}),\ldots,\lf_t(\mathbf{n}))\in A^t$ of elements $\mathbf{n}\in \Zmod{N}^D$ under a homomorphism $\Zmod{N}^D\to\Zmod{N}^t$ given by a system of linear forms $\lf_1,\lf_2,\ldots,\lf_t:\Z^D\to \Z^t$. We are interested in counting such configurations, especially under the weak assumption that the density $\abs{A}/N$ of $A$ in $\Zmod{N}$ is fixed. To this end we set up the following notation.

\begin{defn}[Solution measure]
Let $D,t \geq 1$ be integers and let $\Lf = (\lf_1, \ldots, \lf_t)$ be a system of linear forms $\lf_1, \ldots, \lf_t : \Z^D \to \Z$. For any function $f : \Zmod{N} \to \C$ we write
\begin{align*}
\Sol_\Lf(f) &:= \E_{\mathbf{n} \in \Zmod{N}^D} f\big(\lf_1(\mathbf{n})\big) \cdots f\big(\lf_t(\mathbf{n})\big),
\end{align*}
referring to this as the \emph{solution measure} of $f$ across $\Lf$. 
\end{defn}

When $f$ is the indicator function $1_A$ of a set $A \subset \Zmod{N}$ (that is $f(x) = 1$ if $x \in A$ and $f(x) = 0$ otherwise), the quantity $\Sol_\Lf(A) := \Sol_\Lf(1_A)$ is simply a normalized count of the configurations corresponding to $\Lf$ in $A$. For example, if $\Lf=(n_1, n_1+n_2, n_1+2n_2, n_1+3n_2)$, then $\Sol_\Lf(A)$ is the number of 4-term arithmetic progressions in $A$, divided by $N^2$.\footnote{More generally, one sees easily that $\Sol_\Lf(A) = \abs{A^t \cap \text{Im}\, \Lf}/ \abs{\text{Im}\, \Lf}$, where $\Lf$ denotes the homomorphism $\Zmod{N}^D\to \Zmod{N}^t$ mentioned above (by a slight abuse of notation).}

Such solution counts have been treated in numerous works. The simplest non-trivial case is when the linear forms describe the solution set of a single linear equation. A central example is the system of forms $\threeAP := (n_1, n_1+n_2, n_1+2n_2)$ determining 3-term arithmetic progressions. These are solutions to the equation $x-2y+z=0$, and 
it is a classical result of Roth \cite{Roth} that for any $\alpha>0$ and $N \geq N_0(\alpha)$, any set $A \subset \Zmod{N}$ of density at least $\alpha$ contains a non-trivial 3-term progression (i.e. one with $n_2 \neq 0$). Combined with a short averaging argument of Varnavides \cite{Varnavides}, this in fact implies that $\Sol_\threeAP(A) \geq c(\alpha)$, where $c(\alpha) > 0$ for non-zero $\alpha$. Moreover, Croot \cite{croot:3APminconv} proved that the best possible lower bound behaves nicely for prime moduli $N$, answering a question of Green.

\begin{theorem}[Croot's limit]\label{thm:croot}
Fix $\alpha \in [0,1]$, and for any positive integer $N$ set 
\[ m_\threeAP(\alpha, N) := \min_{ A \subset \Zmod{N},\, \abs{A} \geq \alpha N} \Sol_\threeAP(A). \]
Then $m_\threeAP(\alpha, N)$ converges as $N\to\infty$ through the primes.
\end{theorem}

Croot also showed that $m_\threeAP(\alpha, N)$ can fail to converge if $N$ is allowed to tend to $\infty$ over the odd numbers \cite[Theorem 2]{croot:3APminconv}. This failure comes from integers sharing some fixed factor, so it is natural to address it by restricting $N$ to the primes.

A central tool in Croot's proof was the classical Fourier transform, and his argument can be viewed as an instance of what is now often referred to as the regularity method in arithmetic combinatorics. Various Fourier-analytic versions of this method, consisting roughly in using the dominant Fourier coefficients of an additive set to obtain information on the set's additive structure, have been applied successfully to numerous other combinatorial problems; see for instance \cite{Roth,Bourgain,green-ruzsa:sumfree} and also \cite[Chapter 4]{T-V}. In the last decade, the scope of this method has been considerably widened by the development of a generalization of Fourier analysis known as \emph{higher-order Fourier analysis}, yielding several notable applications \cite{GTarith,Taodic}. 
This paper aims to contribute to this process and illustrate further the applicability of the theory, in connection with Theorem \ref{thm:croot}. Indeed, it is natural to ask whether the result holds more generally for $k$-term progressions with $k\geq 3$, a prospect also raised by Green. While the classical Fourier-analytic regularity method as used by Croot is known not to be helpful for this question, we show here that the higher-order theory does yield the desired generalization. To this end, in particular we shall adapt parts of the work of Green and Tao \cite{GTarith} and combine this with results of Szegedy \cite{SzegedyHFA}. The generalization of Theorem \ref{thm:croot} to longer progressions is then a special case of the following result.

\begin{theorem}\label{mincount}
Fix $\alpha \in [0,1]$ and let $\Lf$ be a system of integer linear forms, any two of which are linearly independent. Set, for every positive integer $N$, 
\[ m_\Lf(\alpha, N) := \min_{A \subset \Zmod{N},\, \abs{A} \geq \alpha N} \Sol_\Lf(A). \]
Then $m_\Lf(\alpha, N)$ converges as $N \to \infty$ through primes.
\end{theorem}
The pairwise linear-independence assumption ensures that the configurations have finite complexity in the sense of \cite{GTlin}. We make some further remarks on this assumption in Appendix \ref{appendix:pairwise-independence}.

In addition to the minimum number of configurations in sets of a given density, another central quantity of interest in this area is the maximum density of a set containing no configurations whatsoever.
\begin{defn}
Given a system $\Lf $ of linear forms $\lf_1, \ldots, \lf_t:\Z^D\to \Z$ and a positive integer $N$, we say $A\subset \Zmod{N}$ is \emph{$\Lf$-free} if $A$ does not contain any configurations determined by $\Lf$, that is if $A^t\cap \Lf(\Zmod{N}^D)=\emptyset$. We define
\[
	d_\Lf(\Zmod{N}) := \max_{A\subset \Zmod{N},\, A\text{ is $\Lf$-free}} \abs{A}/N.
\]
If $\mathcal{F}$ is a finite family of such systems, we say $A\subset \Zmod{N}$ is \emph{$\mathcal{F}$-free} if $A$ is $\Lf$-free for every $\Lf\in\mathcal{F}$, and we define
\[
	d_\mathcal{F}(\Zmod{N}) := \max_{A\subset \Zmod{N},\, A\text{ is $\mathcal{F}$-free}}\ \abs{A}/N.
\]
\end{defn}
Our main result concerning these quantities is the following theorem which extends \cite[Theorem 1.3]{candela-sisask}.
\begin{theorem}\label{maxdensity}
Let $\mathcal{F}$ be a finite family of systems of linear forms, in each of which the forms are pairwise linearly independent. Then $d_\mathcal{F}(\Zmod{N})$ converges as $N\to\infty$ over primes.
\end{theorem}

The quantities $m_\Lf(\alpha) := \lim_{\substack{N\to\infty \\ N\text{ prime}} } m_\Lf(\alpha,N)$ and $d_\mathcal{F} := \lim_{\substack{N\to\infty \\ N\text{ prime}} } d_\mathcal{F}(\Zmod{N})$ stemming from these results depend significantly on whether the systems of forms involved are \emph{invariant}. We say a system $\Lf$ is invariant if $\Lf(\Q^D)$ is invariant under translations by constant vectors, that is if $\Lf(\Q^D)+(1,\ldots,1)=\Lf(\Q^D)$.

Regarding the quantity $m_\Lf(\alpha)$, it is a highly non-trivial fact related to Szemer\'edi's theorem that if $\Lf$ is invariant then an analogue of the Varnavides version of Roth's theorem holds, namely $\Sol_\Lf(A)\geq c_\Lf(\alpha) > 0$ for any set $A\subset\Zmod{N}$ of density at least $\alpha > 0$.\footnote{This can be seen to follow easily from \cite[Theorem 1]{KSVGR}.} In particular, we have $m_\Lf(\alpha)>0$ for every $\alpha>0$. By contrast, if $\Lf$ is not invariant, then there exists $\alpha>0$ such that for each large prime $N$ there is a subset of $\Zmod{N}$ of density at least $\alpha$ with no $\Lf$-configurations whatsoever; see Lemma \ref{lemma:non-invariant-free-sets}. In particular we have $m_\Lf(\alpha')=0$ for all $\alpha' \leq \alpha$. The problem of estimating this function $m_\Lf$ for a general $\Lf$ is of course an extension of the well-known problem of improving the bounds for Szemer\'edi's theorem.

On the other hand, the limit $d_\mathcal{F}$ is of interest mainly for families $\mathcal{F}$ consisting only of non-invariant forms. Indeed, even if one weakens the definition of $\Lf$-free sets to allow them to contain certain trivial configurations, such as constant vectors, it follows from Szemer\'edi's theorem that $d_\mathcal{F} = 0$ whenever $\mathcal{F}$ contains an invariant form. By contrast, for families $\mathcal{F}$ consisting only of non-invariant forms it is easy to see that $d_\mathcal{F} > 0$ (see Lemma \ref{lemma:non-invariant-free-sets}), though not much is known concerning the exact value of this constant. It would be interesting to understand this quantity in general; see \cite{Schoen} for some results of Schoen in this direction. 

Let us now briefly describe the combinatorial ideas underlying the above theorems. Croot's proof of Theorem \ref{thm:croot} consisted essentially in showing that, given an arbitrary set in $\Zmod{p}$, there exists a set in $\Zmod{q}$ having roughly the same solution measure for the system of forms corresponding to 3-term progressions, provided $p$ and $q$ are sufficiently large primes with $q \gg p$. We shall follow the same broad strategy for Theorem \ref{mincount}, and will combine this with a so-called arithmetic removal lemma to obtain Theorem \ref{maxdensity}.
To state the main result underpinning this strategy, we use the following definition.

\begin{defn}[Size of forms]
We say that a system $\Lf = (\lf_1, \ldots, \lf_t)$ of linear forms $\lf_1, \ldots, \lf_t : \Z^D \to \Z$ has \emph{size at most $L$} if $D,t \leq L$ and the coefficients of each $\lf_i$ have absolute value at most $L$.
\end{defn}
Our main combinatorial result can now be stated as follows.

\begin{theorembk}[Periodic transference]
Let $L \geq 1$ be an integer and let $\delta \in (0,1)$. Then for any primes $p, q \geq N_0(\delta,L)$ and any set $A \subset \Zmod{p}$, there is a set $B \subset \Zmod{q}$ such that, for any system $\Lf$ of linear forms of size at most $L$, any two of which are linearly independent, one has $\abs{ \Sol_\Lf(A) - \Sol_\Lf(B) } \leq \delta$.
\end{theorembk}
Note that in particular the densities $\abs{A}/p$ and $\abs{B}/q$ of the sets are very close. This theorem will actually be a simple consequence of the following functional version, where we write $\Sol_\Lf(f:\Zmod{N})$ to emphasize the domain of a function $f$.

\begin{theorembk}[Periodic transference, functional version]\label{pertranfn}
Let $L \geq 1$ be an integer and let $\delta \in (0,1)$. Then for any primes $p, q \geq N_0(\delta,L)$ and any function $f : \Zmod{p} \to [0,1]$, there is a function $f' : \Zmod{q} \to [0,1]$ such that, for any system $\Lf$ of linear forms of size at most $L$, any two of which are linearly independent, one has
\[ \abs{ \Sol_{\Lf}(f : \Zmod{p}) - \Sol_\Lf(f' : \Zmod{q}) } \leq \delta. \]
\end{theorembk}

In fact, both of these results hold for $p$ and $q$ positive integers as long as these do not have small prime factors, so the restriction of $N$ to prime moduli in our main applications can be relaxed somewhat; this is discussed in Section \ref{section:applications}.

To establish Theorem \ref{pertranfn}, we shall use results of Green and Tao \cite{GTarith} and Szegedy \cite{SzegedyHFA} in higher-order Fourier analysis. We shall in fact develop variants of some of these results for periodic nilsequences. One of the novelties in this setting lies in that certain periodic polynomial orbits on nilmanifolds equidistribute in a particularly nice way; this is the content of Proposition \ref{prop:full_fact}. We defer discussion of this to the relevant sections, however, once the appropriate terminology has been introduced.

The paper has the following outline. Section \ref{section:background} provides background on uniformity norms, nilmanifolds, and polynomial sequences. In Section \ref{section:inverse} we record an inverse theorem for the $U^d$ norm for functions on finite cyclic groups, in terms of \emph{periodic} nilsequences, which follows from the main results of Szegedy in \cite{SzegedyHFA}, and we state a corresponding \emph{regularity lemma} for such functions. In Sections \ref{section:counting} and \ref{section:regularity} we develop variants for the periodic setting of the irrational regularity lemmas and counting lemmas of Green and Tao \cite{GTarith}. With these in hand, we shall then need to construct a polynomial nilsequence with prescribed period and equidistribution properties; the construction is presented in Section \ref{section:construction}. In Section \ref{section:transference} the transference result is proved and the combinatorial applications above are finally given in Section \ref{section:applications}. We make some closing remarks in Section \ref{section:remarks}.\\

\textbf{Acknowledgements.} The authors are very grateful to Ben Green for initial conversations that inspired this work. The first-named author was funded in part by the EPSRC and the Fondation Sciences Math\'ematiques de Paris, and the second by the EPSRC and the G\"oran Gustafsson Foundation.

\section{Background notions}\label{section:background}

\subsection{Gowers uniformity norms}

One of the main tools used in this paper is an arithmetic regularity lemma, which decomposes an arbitrary bounded function on $\Zmod{N}$ as a sum of a structured part and some error terms. The sense in which one of these terms constitutes an error is that it is small in a particular \emph{uniformity norm}. These norms can be defined as follows.

\begin{defn}[Uniformity norms]
Let $G$ be a finite abelian group, let $d$ be a positive integer, and let $f : G \to \C$ be a function. We define
\[ \norm{f}_{U^d}^{2^d} := \E_{x \in G,\,h \in G^d} \prod_{\varepsilon \in \{0,1\}^d} \mathcal{C}^{\abs{\varepsilon}} f( x + \varepsilon \cdot h ), \]
where $\mathcal{C}$ is the complex-conjugation operator, $\abs{\varepsilon} = \varepsilon_1 + \cdots + \varepsilon_d$, and $\varepsilon \cdot h = \varepsilon_1 h_1 + \cdots + \varepsilon_d h_d$.
\end{defn}

These norms were introduced by Gowers \cite{GSz}. Their role in arithmetic combinatorics is by now well described in several sources; see for example \cite{GTlin,T-V}. Here we restrict the discussion to the following standard facts. First, these norms are nested: $\norm{f}_{U^d} \leq \norm{f}_{U^{d+1}}$ for any $d \geq 1$. Second, they can be used to control solution measures, in the following sense.

\begin{theorem}\label{thm:GvNdiag}
For any integer $L \geq 1$ there are integers $s = s(L)$ and $C_L$ such that if $\Lf$ is any system of integer linear forms of size at most $L$, any two of which are linearly independent, and $N$ is a positive integer with no prime factors less than $C_L$, then
\[ \big| \Sol_\Lf(f) - \Sol_\Lf(g) \big| \leq L\norm{f-g}_{U^{s+1}} \]
for any functions $f, g : \Zmod{N} \to [0,1]$.
\end{theorem}

This result is tied to a family of results known in this area as generalized von Neumann theorems. The proof of this version is essentially contained in \cite{GTlin} and the result is also discussed in \cite{tao:HFA}. 
We note also the simple bound 
\begin{equation}\label{eqn:sol_L1}
\big| \Sol_\Lf(f) - \Sol_\Lf(g) \big| \leq L\norm{f-g}_{L_1} = L\, \E_{x\in \Zmod{N}} \abs{f(x)-g(x)}
\end{equation}
provided $N$ is prime to at least one coefficient of each form in $\Lf$.

\subsection{Nilmanifolds and polynomial sequences}\label{section:nilmanifolds}

This paper depends heavily on the work of Green and Tao \cite{GTarith,GTOrb} on the quantitative behaviour of polynomial nilsequences. In this subsection we review the basic notation and concepts involved, so as to set this paper in a workable context, but we omit several details, for which we refer the reader to \cite{GTarith,GTOrb}.

\begin{defn}[Filtrations]
Let $G$ be a group. We call a sequence $G_\bullet = (G_i)_{i \geq 0}$ of subgroups of $G$ a \emph{filtration} on $G$ of degree at most $s$ if
	\[ G = G_0 = G_1 \supset G_2 \supset \cdots \supset G_s \supset G_{s+1} = G_{s+2} = \cdots = \{ \id_G \} \]
and $[G_i,G_j] \subset G_{i+j}$ for all $i,j \geq 0$. Here $[g,h] := g h g^{-1} h^{-1}$ denotes the group commutator of $g,h \in G$ and $[A,B]$ denotes the subgroup of $G$ generated by 
	$\{ [a,b]: a \in A, b \in B \}$.
\end{defn}

\begin{defn}[Nilmanifolds]
If $G$ is a connected, simply-connected nilpotent Lie group and $\Gamma$ is a discrete, cocompact subgroup, we call $G/\Gamma$ a \emph{nilmanifold}. If $G_\bullet$ is a filtration on $G$ of degree at most $s$, and the $G_i$ are closed and connected with the subgroups $\Gamma_i := \Gamma \cap G_i$ cocompact in $G_i$, then we call the pair $(G/\Gamma, G_\bullet)$ a \emph{filtered nilmanifold} of degree at most $s$. We define the total dimension of such a nilmanifold to be the quantity $\sum_{i=0}^s \dim G_i$.
\end{defn}

Throughout the paper we shall write $m$ for the dimension of $G$ and $m_i$ for the dimension of $G_i$, whenever it is obvious from the context to which groups we are referring.

We also need the notion of a \emph{Mal'cev basis} for a filtered nilmanifold $(G/\Gamma, G_\bullet)$. This notion was introduced in \cite{Malcev}, and it is defined and discussed in Appendix \ref{appendix:Malcev} here. For now we note only a few salient facts. Such a basis provides a real-coordinate system on $G$ that is consistent with $\Gamma$ and $G_\bullet$, by means of the associated \emph{Mal'cev coordinate map} $\psi : G \to \R^m$, a diffeomorphism for which 
\begin{enumerate}
	\item $\psi(\Gamma) = \Z^m$,
	\item $\psi(G_i) = \{0\}^{m-m_i} \times \R^{m_i}$, and
	\item $\psi^{-1}([0, 1)^m)\subset G$ is a fundamental domain for $G/\Gamma$, that is for any $g\in G$ there exists a unique element of $\Gamma$, denoted $[g]$, such that the element $\{g\}:= g\, [g]^{-1}$ satisfies $\psi( \{g\} )\in [0,1)^m$.
\end{enumerate}
Thus an element of $G$ lies in $\Gamma$ if and only if all its coordinates are integers, and in $G_i$ if and only if its first $m-m_i$ coordinates are $0$.
These coordinates are useful in many ways, for example in classifying certain homomorphisms on $G$ and in defining a notion of distance on the nilmanifold (see Appendix \ref{appendix:Malcev}).

\begin{defn}[Complexity of a filtered nilmanifold]
	Let $(G/\Gamma, G_\bullet)$ be a filtered nilmanifold of degree at most $s$, and let $\X$ be a Mal'cev basis for $G/\Gamma$ adapted to $G_\bullet$. We say that $(G/\Gamma,G_\bullet,\X)$ has \emph{complexity} at most $M$ if $m$, $s$ and the rationality (see \cite[Definition 2.4]{GTOrb}) of $\X$ are all at most $M$. 
\end{defn}

In this paper a filtered nilmanifold will always come with a Mal'cev basis, but the basis may sometimes not be specified explicitly when it is clear from the context.

We also need the notion of a subnilmanifold. Recall that a rational number is said to have \emph{height} $M$ if it equals $a/b$ with $a,b$ coprime and $\max(\abs{a},\abs{b}) = M$. Recall also that a subgroup $G'$ of $G$ is said to be a \emph{rational subgroup} if $\Gamma\cap G'$ is a cocompact subgroup of $G'$ \cite{GTOrb}. We say that such a subgroup $G'$ is $M$-rational, or has complexity at most $M$ in $G$ (relative to $\X$), if the Lie algebra $\mathfrak{g'}$ of $G'$ is generated by linear combinations of the elements of $\X$ with rational coefficients of height at most $M$.

\begin{defn}[Subnilmanifolds]
	Given a filtered nilmanifold $(G/\Gamma,G_\bullet,\X)$ of degree at most $s$, a \emph{subnilmanifold of} $G/\Gamma$ \emph{of complexity at most} $M$ is a filtered nilmanifold $(G'/\Gamma',G_\bullet',\X')$ of complexity at most $M$ with each subgroup $G_i'$ in $G_\bullet'$ being a rational subgroup of $G_i$ of complexity at most $M$, where $\Gamma'= G'\cap\Gamma$, and where each element of the Mal'cev basis $\X'$ is a linear combination of elements of $\X$ with rational coefficients of height at most $M$.
\end{defn}

\begin{defn}[Polynomial sequences]
Given a filtration $G_\bullet$ of degree at most $s$ on a group $G$, we define $\poly(\Z, G_\bullet)$ to consist of all maps $g : \Z \to G$ such that
\[ \partial_{h_i} \cdots \partial_{h_1}g (n) \in G_i \quad \text{for all $i \geq 0$ and $h_1, \ldots, h_i, n \in \Z$,} \]
where $\partial_h$ is the difference operator given by $\partial_h g(n) := g(n+h)g(n)^{-1}$. We call any such map $g$ a \emph{polynomial sequence}, or simply a \emph{polynomial}. 
\end{defn}

A very useful and non-trivial fact about $\poly(\Z, G_\bullet)$ is that it forms a group under pointwise multiplication. This is referred to as the Lazard-Leibman theorem in \cite{GTarith}; see that paper and \cite{Leibman:poly_seqs, Leibman:poly_maps} for further details and references. We shall generally use this fact without mention. One also has quite a tangible description of polynomials via the following lemma.

\begin{lemma}[Taylor expansion]\label{lemma:Taylor}
Let $g \in \poly(\Z, G_\bullet)$, where $G_\bullet$ has degree at most $s$. Then there are unique \emph{Taylor coefficients} $g_i \in G_i$ such that
\[ g(n) = g_0 g_1^n g_2^{\binom{n}{2}} \cdots g_s^{\binom{n}{s}} \]
for all $n \in \Z$, and, conversely, every such expression represents a polynomial sequence $g \in \poly(\Z,G_\bullet)$. Moreover, if $H$ is a subgroup of $G$ and $g$ is $H$-valued then we have $g_i \in H$ for each $i$.\footnote{This final claim is of course essentially a generalization of the fact that a polynomial $p : \Z \to \R$ is integer-valued if and only if its Newton series $p(n) = a_0 + a_1 n + \cdots + a_s \binom{n}{s}$ has integer coefficients.}
\end{lemma}
\begin{proof}
Except for the final claim, this is contained in \cite[Lemma A.1]{GTarith}; we also give a proof of a slight generalization in Appendix \ref{appendix:Taylor}. Note that the $g_i$ may be found inductively by $g_0 := g(0)$, $g_1 := g_0^{-1} g(1)$, $g_j := (g_0 g_1^j \cdots g_{j-1}^{\binom{j}{j-1}})^{-1} g(j)$, from which the final claim is clear.
\end{proof}

\begin{defn}[Nilsequences]
	We call a function $f : \Z \to \C$ a (\emph{polynomial}) \emph{nilsequence} of degree at most $s$ and complexity at most $M$ if there is a nilmanifold $(G/\Gamma, G_\bullet, \X)$ of degree at most $s$ and complexity at most $M$, together with a polynomial $g \in \poly(\Z, G_\bullet)$ and a Lipschitz function $F : G/\Gamma \to \C$ with $\norm{F}_{\Lip(\X)} \leq M$, such that $f(n) = F(g(n)\Gamma)$ for all $n \in \Z$.
\end{defn}

The Lipschitz norm is defined here in terms of a metric $d_{G/\Gamma} = d_{G/\Gamma,\X}$ on $G/\Gamma$ by
\[ \norm{F}_{\Lip(\X)} := \norm{F}_\infty + \sup_{x\Gamma \neq y\Gamma \in G/\Gamma} \frac{ \abs{ F(x\Gamma) - F(y\Gamma) } }{ d_{G/\Gamma}(x\Gamma, y\Gamma) }. \]
The metric structure on $G/\Gamma$ comes from a metric $d_G = d_\X$ on $G$, defined to be the largest right-invariant metric on $G$ for which the distance from $x$ to the identity is at most $\norm{\psi(x)}_\infty$, i.e. is bounded by its largest coordinate in absolute value. The distance $d_{G/\Gamma}(x\Gamma, y\Gamma)$ is then defined to be the infimum of $d_\X(x',y')$ over all representatives $x' \in x\Gamma$, $y' \in y\Gamma$. See \cite[Definition 2.2]{GTOrb} for more details.

Finally, we need a definition for our periodic setting.

\begin{defn}
	Let $(G/\Gamma,G_\bullet)$ be a filtered nilmanifold, and let $N$ be a positive integer. We say a sequence $g \in \poly(\Z, G_\bullet)$ is \emph{$N$-periodic mod $\Gamma$} if $g(n+N)\Gamma = g(n)\Gamma$ for all $n \in \Z$. Occasionally we may drop the mention of the period and $\Gamma$, and simply refer to a polynomial as being periodic. We say a nilsequence $F(g(n)\Gamma)$ (or an orbit $(g(n)\Gamma)$) is \emph{$N$-periodic} if its associated polynomial $g$ is $N$-periodic mod $\Gamma$. Finally, we call an element $h\in G$ an \emph{$N$th root mod} $\Gamma$ if $g^N\in\Gamma$. 
\end{defn}

\section{A periodic inverse theorem}\label{section:inverse}

In recent years it has been a central objective in higher-order Fourier analysis to obtain a general result for the $U^d$ norms known as an inverse theorem. Roughly speaking, in one of its most useful forms this result should characterize a function on $[N]:=\{1,2,\ldots,N\}$ having non-trivial $U^d$ norm as one having non-trivial correlation with some $d-1$ step nilsequence of bounded complexity. Such a result was finally established by Green, Tao and Ziegler in \cite{GTZ}. An alternative approach to this inverse theorem was given by Szegedy in \cite{SzegedyHFA}, using the theory of nilspaces developed by Camarena and Szegedy in \cite{Cam-Szeg}, itself inspired by fundamental work of Host and Kra \cite{HK}. The main results in \cite{SzegedyHFA} yield an inverse theorem for functions on a finite cyclic group, involving \emph{periodic} nilsequences, which is crucial for this paper. (However, see \S \ref{subs:Manners} for an update concerning these matters.) To state the inverse theorem we use the following notion.

\begin{defn}\label{defn:char0}
Let $p_1(N)$ denote the smallest prime factor of a positive integer $N$ (with $p_1(1) := 1$). We say an infinite set $\mathcal{N}$ of positive integers has \emph{characteristic} 0 if $p_1(N)\to\infty$ as $N\to\infty$ through $\mathcal{N}$. We say a sequence of finite abelian groups $(A_i)_{i\in \N}$ of increasing size has characteristic 0 if $\{\abs{A_i}: i\in \N\}$ has characteristic 0.
\end{defn}

\begin{remark}
It is clear that $\mathcal{N}$ has characteristic 0 if and only if for any integer $n>1$, only finitely many $N$ in $\mathcal{N}$ are divisible by $n$. Thus a sequence $(A_i)_{i\in \N}$ of characteristic 0 as above forms a group-family of characteristic 0 in the sense of \cite{SzegedyHFA}.
\end{remark}
 
The version of the inverse theorem that we shall use is the following.

\begin{theorem}[Periodic inverse theorem]\label{thm:perinverse}
Let $\cN\subset \N$ be a set of characteristic $0$, let $s$ be a positive integer, and let $\delta > 0$. There exists $M>0$ such that if $\norm{f}_{U^{s+1}(\Zmod{N})} \geq \delta$ for some function $f : \Zmod{N} \to \C$ with $\norm{f}_\infty \leq 1$ and $N\in \cN$, then there exists an $N$-periodic polynomial nilsequence $h$ of degree at most $s$ and complexity at most $M$ such that $\E_{n\in \Zmod{N}} f(n) \overline{h(n)} \geq c_s(\delta)>0$.
\end{theorem}

This theorem follows from (the proof of) \cite[Theorem 10]{SzegedyHFA}. (Note that from the proof of the latter theorem in \cite{SzegedyHFA} one indeed gets that the polynomial underlying the nilsequence $h$ is $N$-periodic mod $\Gamma$). 

By the same arguments as in \cite[Section 2]{GTarith}, using in particular that the sum or product of two $N$-periodic nilsequences is again an $N$-periodic nilsequence, we may deduce the following arithmetic regularity lemma.

\begin{theorembk}[Periodic, non-irrational arithmetic regularity lemma]\label{thm:periodic_regularity_non_irrational}
Let $\cN\subset \N$ be a set of characteristic $0$, let $s \geq 1$ be an integer, let $\epsilon > 0$, and let $\mathcal{F} : \R^+ \to \R^+$ be a growth function. There exists $M >0$ such that for any $N\in \cN$ and any function $f : \Zmod{N} \to [0,1]$ there is a decomposition
\[ f = f_\nil + f_\sml + f_\unf \]
of $f$ into functions $f_* : \Zmod{N} \to [-1,1]$ such that
\begin{enumerate}
\item $f_\nil$ is an $N$-periodic nilsequence of degree at most $s$ and complexity at most $M$,
\item $\norm{f_\sml}_2 \leq \epsilon$,
\item $\norm{f_\unf}_{U^{s+1}} \leq 1/\mathcal{F}(M)$, and
\item $f_\nil$ and $f_\nil + f_\sml$ take values in $[0,1]$. 
\end{enumerate}
\end{theorembk}

This regularity lemma essentially allows us to reduce the study of $\Sol_\Lf(f)$ to that of $\Sol_\Lf(f_\nil)$, this being useful since we have more structural information about $f_\nil$.
Much as noted in \cite{GTarith}, however, it turns out that we shall need stronger information still on $f_\nil$: we shall require the orbit underlying the nilsequence to be highly \emph{irrational}, a quantitative property which guarantees that certain higher-dimensional variants of the orbit are equidistributed. We develop the tools we need for this in the next section.

\section{Irrationality and the periodic counting lemma}\label{section:counting}

Using the regularity lemma, from a function $f$ on $\Zmod{N}$ we obtain a nilsequence $f_\nil(n) = F(g(n)\Gamma)$, where the polynomial sequence $g: \Z \to G$ is $N$-periodic mod $\Gamma$. In \cite{GTOrb}, Green and Tao developed powerful machinery to understand polynomial orbits quantitatively, and especially when such orbits are equidistributed. This machinery was built upon in \cite{GTarith}, where a notion of \emph{irrationality} was introduced that is useful for dealing with solution measures across linear forms. In particular, if a polynomial sequence $g$ is highly irrational, then $\Sol_\Lf(F(g(\cdot)\Gamma))$ is very close to a certain integral involving $F$ that is essentially independent of $g$. A generalization of this is what is called a \emph{counting lemma} in \cite{GTarith}. We shall use a slight weakening of this notion of irrationality. Before we define this formally, which will take some preparation, we state the corresponding counting lemma by way of motivation. For this we use the following notation: given a sequence $g : \Z \to G$ and a system $\Lf=(\lf_1,\ldots,\lf_t)$ of linear forms $\lf_i:\Z^D\to \Z$, we write
\[ g^\Lf(\mathbf{n}) := \big(g(\lf_1(\mathbf{n})), \ldots, g(\lf_t(\mathbf{n}))\big) \]
for $\mathbf{n} \in \Z^D$ (or $\mathbf{n} \in \Zmod{N}^D$), and we write $G^\Lf/\Gamma^\Lf \subset G^t/\Gamma^t$ for the so-called \emph{Leibman nilmanifold} associated with $(G/\Gamma, G_\bullet, \X)$ and $\Lf$.\footnote{This is defined in \cite{GTarith}; we do not need any information about it beyond its occurrence in Theorem \ref{thm:periodic_counting}.}

\begin{theorem}[Periodic counting lemma]\label{thm:periodic_counting}
Let $M,D,s,t$ be integers with $1\leq D,s,t\leq M$, and let $(G/\Gamma,G_\bullet,\X)$ be a filtered nilmanifold of degree at most $s$ and complexity at most $M$. Let $g\in \poly(\Z,G_\bullet)$ be $A$-irrational and $N$-periodic mod $\Gamma$, for some positive integer $N$. Let $\Lf$ be a system of linear forms $\lf_1,\ldots,\lf_t:\Z^D\to \Z$ with coefficients of magnitude at most $M$.
Then, for any Lipschitz function $F:(G/\Gamma)^t\to \C$ with $\norm{F}_{\Lip(\X^t)} \leq M$, one has
\begin{equation}\label{count-eqn_new}
	\E_{\mathbf{n}\in \Zmod{N}^D} F(g^{\Lf}(\mathbf{n})\Gamma^t) = \int_{g(0)^\Delta G^\Lf/\Gamma^\Lf} F \quad +\quad o_{A \to \infty;M}(1).
\end{equation}
\end{theorem}
Here, as in \cite{GTarith}, $g(0)^\Delta$ denotes the element $(g(0), \ldots , g(0))\in G^t$, and the integral is with respect to the normalized Haar measure on the coset $g(0)^\Delta G^\Lf/\Gamma^\Lf$. When $F$ has the form $F(x_1,\ldots,x_t) = F'(x_1) \cdots F'(x_t)$, the left-hand side of \eqref{count-eqn_new} is $\Sol_\Lf(F'(g(\cdot)\Gamma))$. The norm $\norm{F}_{\Lip(\X^t)}$ is then easily related to $\norm{F'}_\X$, using the basis $\X^t$ for $(G/\Gamma)^t$ implicit in the above theorem; this is all detailed in Appendix \ref{appendix:Malcev}.

Theorem \ref{thm:periodic_counting} is an adaptation of \cite[Theorem 1.11]{GTarith} to periodic orbits. While the periodicity assumption is not present in \cite{GTarith}, our assumption of $A$-irrationality is somewhat weaker than that of $(A,N)$-irrationality used in that paper (see below), and the error term in \eqref{count-eqn_new} is independent of $N$. We shall deduce the above theorem from \cite[Theorem 1.11]{GTarith}, but first we lay the groundwork for the definition of irrationality. 
Much of this groundwork follows \cite{GTarith}, but we include the details as these are important for later results. Recall that $\Gamma_i := \Gamma \cap G_i$.

\begin{defn}[$i$th-level characters]
Let $(G/\Gamma, G_\bullet)$ be a filtered nilmanifold of degree at most $s$. 
We write $\Ghor_i$ for the group generated by $G_{i+1}$ and $[G_j,G_{i-j}]$ for $0 \leq j \leq i$. 
An \emph{$i$th-level character} is a continuous homomorphism from $G_i$ to $\R$ which vanishes on $\Ghor_i$ and is $\Z$-valued on $\Gamma_i$. We say that such a character is \emph{non-trivial} if it is non-constant.
\end{defn}

\begin{remark}
Note that $\Ghor_i$ is contained in $G_i$ by the filtration property, and is a normal subgroup of $G$ since each $G_j$ is. Observe that for the lower central series filtration, these concepts are only interesting for $i = 1$, as for this filtration we have $\Ghor_i \supset [G_1, G_{i-1}] = G_i$ for $i \geq 2$. Note also that 1st-level characters are precisely the (lifts of) horizontal characters in the sense of \cite{GTOrb}. In \cite{GTarith} $i$th-level characters are called $i$-horizontal characters.
\end{remark}

We next recall the notion of complexity for $i$th-level characters, which is defined in terms of Mal'cev bases (see Appendix \ref{appendix:Malcev} for the definition of these and their corresponding coordinate maps). Given a nilmanifold $(G/\Gamma, G_\bullet, \X)$, we write $\psi_i$ for the restriction of the Mal'cev coordinate map $\psi: G \to \R^m$ to the $i$th-level part of $G$, that is $\psi_i : G_i \to \R^{r_i}$ is $\psi$ composed with the projection to coordinates indexed between $m-m_i+1$ and $m-m_{i+1}$, where, as usual, $m = \dim{G}$ and $m_i = \dim{G_i}$, and $r_i := m_i - m_{i+1}$. The following lemma describing $i$th-level characters using this map is straightforward to verify.

\begin{lemma}[Frequency vector]\label{lemma:frequency_vector}
With the notation above, any $i$th-level character $\xi_i : G_i \to \R$ has the form $\xi_i(g)=k\cdot \psi_i(g)$ for some $k \in \Z^{r_i}$.
\end{lemma}

\begin{defn}[Complexity]
We define the \emph{complexity} of an $i$th-level character $\xi_i$ (relative to $\X$) to be $\abs{k_1}+\cdots +\abs{k_{r_i}}$ for the corresponding $k\in \Z^{r_i}$.
\end{defn}

We can now recall the definition of $(A,N)$-irrationality from \cite{GTarith}.

\begin{defn}[$(A,N)$-irrationality]\label{defn:ANirrational}
Let $(G/\Gamma, G_\bullet, \X)$ be a filtered nilmanifold of degree at most $s$, and let $A,N>0$. An element $g_i \in G_i$ is \emph{$(A,N)$-irrational} if for every non-trivial $i$th-level character $\xi_i$ of complexity at most $A$ we have $\|\xi_i(g_i) \|_{\R/\Z}\geq A/N^i$. A polynomial $g \in \poly(\Z,G_\bullet)$ is \emph{$(A,N)$-irrational} if its Taylor coefficient $g_i$ is $(A,N)$-irrational for each $i \in [s]$.
\end{defn}

The main definition of this section, then, is the following variant. 

\begin{defn}[$A$-irrationality]\label{defn:irrational}
Let $(G/\Gamma, G_\bullet, \X)$ be a filtered nilmanifold of degree at most $s$. 
 We say that $g_i \in G_i$ is \emph{$A$-irrational} if for every non-trivial $i$th-level character $\xi_i$ of complexity at most $A$ we have $\xi_i(g_i) \notin \Z$. We say that a polynomial $g \in \poly(\Z,G_\bullet)$ is \emph{$A$-irrational} if its Taylor coefficient $g_i$ is $A$-irrational for each $i \in [s]$. In other words, a polynomial is $A$-irrational if and only if it is $(A,N)$-irrational for some $N$.
\end{defn}

Thus, requiring $g$ to be irrational amounts to requiring the Mal'cev coordinates of each of its Taylor coefficients not to satisfy any linear equation (mod $1$) with small integer coefficients. This variant is relevant in particular for periodic polynomials, as we shall now see by finally deducing our periodic counting lemma rather simply from the counting lemma of Green and Tao.\footnote{Although it may seem strange to talk about sequences that are both periodic and irrational, note that we only use these terms in the quantitative sense; a polynomial sequence can indeed be $A$-irrational and $N$-periodic mod $\Gamma$ as long as $N$ is sufficiently large in terms of $A$.}

\begin{proof}[Proof of Theorem \ref{thm:periodic_counting}]
Restrict \cite[Theorem 1.11]{GTarith} to $P = [N]^D$ and divide both sides of the conclusion by $N^D$ to get that for any $(A,N)$-irrational polynomial $h : \Z \to G$ with $h(0) = \id_G$ we have
	\[ \E_{\mathbf{n}\in [N]^D} F(h^{\Lf}(\mathbf{n})\Gamma^t) = \int_{g(0)^\Delta G^\Lf/\Gamma^\Lf} F \quad +\quad o_{A \to \infty;M}(1) + o_{N \to \infty; M}(1). \]
Now, by assumption the polynomial $g$ given to us is $(A,kN)$-irrational for all large enough integers $k$, and so 
	\begin{equation}\label{multiple_counting} \E_{\mathbf{n}\in [kN]^D} F(g^{\Lf}(\mathbf{n})\Gamma^t) = \int_{g(0)^\Delta G^\Lf/\Gamma^\Lf} F \quad +\quad o_{A \to \infty;M}(1) + o_{kN \to \infty; M}(1). \end{equation}
Decomposing $[kN]^D = \{0,N,\ldots,(k-1)N\}^D \oplus [N]^D$, the left-hand side of \eqref{multiple_counting} is
\[ \E_{\mathbf{n_0} \in \{0,N,\ldots,(k-1)N\}^D} \E_{\mathbf{n}\in [N]^D} F(g^{\Lf}(\mathbf{n}+\mathbf{n_0})\Gamma^t) = \E_{\mathbf{n}\in [N]^D} F(g^{\Lf}(\mathbf{n})\Gamma^t), \]
the last equality being a consequence of the periodicity assumption on $(g(n)\Gamma)$. Letting $k$ tend to infinity in \eqref{multiple_counting}, we then obtain \eqref{count-eqn_new}.
\end{proof}

\section{A factorization theorem and a strengthened regularity lemma}\label{section:regularity}

Given the relevance of Theorem \ref{thm:periodic_counting} to solution measures of nilsequences, we now aim to strengthen our regularity lemma by adding the property of irrationality to the structured part. A similarly strengthened regularity lemma, in which the structured part takes the form of a so-called \emph{virtual} nilsequence, was \cite[Theorem 1.2]{GTarith}, one of the main results of that paper. In this section we show that in the periodic setting, when the period is prime, or more generally when it belongs to a given set of characteristic 0, we can do away with this notion of virtual nilsequences, obtaining the following result.

\begin{theorembk}[Periodic regularity lemma with irrational nilsequence\footnote{See also \S\ref{subs:Manners} for an update.}]\label{thm:periodic_regularity}
Let $\cN\subset \N$ be a set of characteristic $0$, let $s \geq 1$ be an integer, let $\epsilon > 0$, and let $\mathcal{F} : \R^+ \to \R^+$ be a growth function. Then there is a number $M = O_{s,\epsilon,\mathcal{F},\cN}(1)$ such that for any $N\in \cN$ with $p_1(N) \geq N_0(s, \epsilon, \mathcal{F}, \cN)$ and any function $f : \Zmod{N} \to [0,1]$ there is a decomposition
\[ f = f_\nil + f_\sml + f_\unf \]
of $f$ into functions $f_* : \Zmod{N} \to [-1,1]$ such that
\begin{enumerate}
\item $f_\nil$ is an $N$-periodic, $\mathcal{F}(M)$-irrational nilsequence of degree at most $s$ and complexity at most $M$,
\item $\norm{f_\sml}_2 \leq \epsilon$,
\item $\norm{f_\unf}_{U^{s+1}} \leq 1/\mathcal{F}(M)$, and\label{reg:unf}
\item $f_\nil$ and $f_\nil + f_\sml$ take values in $[0,1]$. 
\end{enumerate}
\end{theorembk}

The key to this strengthening of the regularity lemma is the following equidistribution result for prime-periodic orbits, which should be compared to \cite[Lemma 2.10]{GTarith}. 

\begin{propositionbk}[Factorization of periodic polynomials]\label{prop:full_fact}
Let $(G/\Gamma,G_\bullet, \X)$ be a filtered nilmanifold of degree at most $s$ and complexity at most $M_0$, and let $\mathcal{F}$ be a growth function. Let $g \in \poly(\Z,G_\bullet)$ satisfy $g(0)=\id_G$ and $g(q\Z)\subset \Gamma$, where $p_1(q) > O_{M_0,\mathcal{F}}(1)$. Then for some $M\in [M_0,O_{M_0,\mathcal{F}}(1)]$ there is a factorization $g= g'\,\gamma$ with the following properties.
\begin{enumerate}
	\item $g'\in \poly(\Z,G'_\bullet)$ is $\mathcal{F}(M)$-irrational, for a subnilmanifold $(G'/\Gamma',G'_\bullet, \X')$ of $G/\Gamma$ of complexity $O_M(1)$, $g'(0)=\id_G$, and $g'(q\Z)\subset \Gamma'$.
\item $\gamma \in \poly(\Z,G_\bullet)$ is $\Gamma$-valued.
\end{enumerate}
In particular, if $g$ is $q$-periodic mod $\Gamma$ then we have $g=\varepsilon\, g'\,\gamma$, where $\varepsilon=\{g(0)\}$ is a constant, $g'$ and $\gamma$ have the above properties and furthermore $g'$ is $q$-periodic mod $\Gamma'$.
\end{propositionbk}

The most important feature of this result, relative to \cite[Lemma 2.10]{GTarith}, is that the equidistribution (or irrationality) of the orbit automatically takes place on a connected subnilmanifold of $G/\Gamma$, rather than what is essentially several connected components. We remark upon some other interesting features at the end of this section, turning now instead to the proof. Much as in \cite[Section 2]{GTarith}, this will build up the factorization in stages. We begin with some simple lemmas.

\begin{lemma}
\label{scaling_lemma}
Let $g \in \poly(\Z, G_\bullet)$ and let $q \in \Z$. Define $h \in \poly(\Z, G_\bullet)$ by $h(n) = g(qn)$. Then, for any $i$, the $i$th Taylor coefficient $h_i$ of $h$ satisfies $h_i = g_i^{q^i} \mod \Ghor_i$, where $g_i$ is the $i$th Taylor coefficient of $g$.
\end{lemma}
\begin{proof}
This follows from the proof of \cite[Lemma A.8]{GTarith}, consisting of many applications of (applications of) the Baker--Campbell--Hausdorff formula.
\end{proof}

\begin{lemma}\label{newton_lemma}
Let $g \in \poly(\Z, G_\bullet)$, let $q \in \Z$ and suppose $g(q\Z) \subset \Gamma$. Then for each $i \geq 0$ there is some $\gamma_i \in \Gamma_i$ such that $g_i^{q^i} = \gamma_i \mod \Ghor_i$, where $g_i$ is the $i$th Taylor coefficient of $g$.
\end{lemma}
\begin{proof}
Set $h(n) = g(qn)$. This is $\Gamma$-valued by assumption, and so its Taylor expansion is $h(n) = \gamma_0 \gamma_1^n \cdots \gamma_s^{\binom{n}{s}}$ for some $\gamma_i \in \Gamma_i$, by Lemma \ref{lemma:Taylor}. Applying Lemma \ref{scaling_lemma}, we are done.
\end{proof}

Thus, modulo $\Ghor_i$, the $i$th Taylor coefficient of a $q$-periodic sequence with trivial constant term is a $q^i$th root mod $\Gamma_i$. The following lemma, an analogue of \cite[Lemma A.7]{GTarith}, is then the technical heart of this section.

\begin{lemma}[Taylor-coefficient factorization]\label{lemma:Taylor_fact}
Let $(G/\Gamma,G_\bullet, \X)$ be a filtered nilmanifold of degree at most $s$, let $A>0$, and suppose $g\in \poly(\Z,G_\bullet)$ satisfies $g(q\Z)\subset \Gamma$ for some integer $q$ with $p_1(q)>A$. If $g$ is not $A$-irrational then there is an index $i\in [s]$ such that its $i$th Taylor coefficient $g_i$ factors as $g_i'\gamma_i$, where $\gamma_i \in \Gamma_i$ and $g_i'$ lies in the kernel of some non-trivial $i$th-level character of complexity at most $A$.
\end{lemma}
\begin{proof}
By definition of $A$-irrationality, there is some $i \in [s]$ for which there is a non-trivial $i$th-level character $\xi_i : G_i \to \R$ with $\xi_i(g_i) \in \Z$. Let $k \in \Z^{r}$, $r := m_i - m_{i+1}$, be the frequency vector for $\xi_i$ given by Lemma \ref{lemma:frequency_vector}, so that
\[ \xi_i(g) = k\cdot \psi_i(g) \quad \text{for each $g \in G_i$}. \]
Now consider $\xi_i(g_i^{q^i})$. On one hand, this equals $q^i \xi_i(g_i) \in q^i\, \Z$. On the other, by Lemma \ref{newton_lemma} we have $g_i^{q^i} = \gamma \mod \Ghor_i$ for some $\gamma \in \Gamma_i$, and since $\xi_i$ annihilates $\Ghor_i$ we have
\[ \xi_i(g_i^{q^i}) = \xi_i(\gamma) = k \cdot \psi_i(\gamma) \in \hcf(k)\, \Z, \]
$\psi_i(\gamma)$ lying in $\Z^r$ by the definition of Mal'cev bases (see Appendix \ref{appendix:Malcev}). Hence $q^i \xi_i(g_i) \in \lcm( q^i, \hcf(k) ) \Z = \hcf(k)q^i \Z$, since $\hcf(k) \leq A$ and $q$ has no prime factors less than $A$ by assumption. Thus $\xi_i(g_i) \in \hcf(k)\, \Z$, and so there is some vector $t \in \Z^r$ such that $k \cdot t = \xi_i(g_i)$. Letting $\gamma_i = \psi_i^{-1}(t)$, that is $\gamma_i = \exp(t_1 X_{m-m_i+1}) \cdots \exp(t_r X_{m-m_{i+1}})$, and setting $g_i' = g_i \gamma_i^{-1}$, we obtain the result.
\end{proof}

The deduction of Proposition \ref{prop:full_fact} from this lemma is completely analogous to the deduction of \cite[Lemma 2.10]{GTarith} from \cite[Lemma A.7]{GTarith}, so we defer it to Appendix \ref{appendix:factorization}. We now show how Proposition \ref{prop:full_fact} implies Theorem \ref{thm:periodic_regularity}.

\begin{proof}[Proof of Theorem \ref{thm:periodic_regularity}]
We begin by applying Theorem \ref{thm:periodic_regularity_non_irrational} to $f$ with a function $\mathcal{F}_0$ that grows sufficiently quickly compared to $\mathcal{F}$, obtaining a decomposition
\[ f = f_\nil + f_\sml + f_\unf \]
and an integer $M_0 = O_{s,\epsilon,\mathcal{F}_0,\cN}(1)$ such that 
\begin{enumerate}
	\item $f_\nil : \Zmod{N} \to [0,1]$ is given by $f_\nil(n) = F_0(g(n)\Gamma)$ for some Lipschitz function $F_0 : G/\Gamma \to \C$ on a nilmanifold $(G/\Gamma, G_\bullet, \X)$ of degree at most $s$ and complexity at most $M_0$, $\norm{F_0}_{\Lip(\X)} \leq M_0$, and $g\in \poly(\Z,G_\bullet)$ is $N$-periodic mod $\Gamma$,
	\item $\norm{f_\sml}_2 \leq \epsilon$, and
	\item $\norm{f_\unf}_{U^{s+1}} \leq 1/\mathcal{F}_0(M_0)$,
\end{enumerate}
as well as the other properties in that theorem. We then apply Proposition \ref{prop:full_fact} with another growth function $\mathcal{F}_1$ (assuming $p_1(N)> O_{M_0,\mathcal{F}_1}(1)$) to obtain a number $M_1 \in [M_0, O_{M_0, \mathcal{F}_1}(1)]$ and a polynomial $g'\in \poly(\Z,G'_\bullet)$ that is $\mathcal{F}_1(M_1)$-irrational and $N$-periodic mod $\Gamma'$ in some subnilmanifold $(G'/\Gamma',G'_\bullet, \X')$ of $G/\Gamma$ of complexity $O_{M_1}(1)$, and satisfies $\varepsilon \,g'(n)\Gamma = g(n)\Gamma$.

The nilsequence of the conclusion then consists of the function $F : G'/\Gamma' \to \C$ given by $F(x\Gamma') := F_0(\varepsilon x\Gamma)$, which has $\norm{F}_{\Lip(\X')} = O_{M_1}(1)$ by \cite[Lemmas A.5 and A.17]{GTOrb}, the nilmanifold $(G'/\Gamma', G_\bullet', \X')$, which has complexity at most $O_{M_1}(1)$, and the polynomial $g'$. Since
\[ f_\nil(n) = F_0( g(n)\Gamma) = F( g'(n) \Gamma') \]
we see that the nilsequence $f_\nil$ has complexity $M \leq \mathcal{C}(M_1)$ relative to these data, for some growth function $\mathcal{C}$. We thus pick $\mathcal{F}_1(x) := \mathcal{F}(\mathcal{C}(x))$ so that $g'$ is $\mathcal{F}(M)$-irrational. In order to ensure part (\ref{reg:unf}) of the conclusion, it suffices to pick $\mathcal{F}_0$ so that $\mathcal{F}_0(M_0) \geq \mathcal{F}(M)$, which we can do since $M = O_{M_1}(1) = O_{M_0, \mathcal{F}}(1)$. Finally, of course $M = O_{s,\epsilon,\mathcal{F},\cN}(1)$, and we are done.
\end{proof}

\begin{remark}\label{rmk:A-rational-fact}
From failure of $A$-irrationality alone in Lemma \ref{lemma:Taylor_fact}, that is with no periodicity-related assumption on $g$, one can still deduce a factorization $g_i = g_i' \gamma_i$, but with $\gamma_i$ being $A$-rational instead of actually lying in $\Gamma$, meaning that $\gamma_i^t \in \Gamma$ for some $1 \leq t \leq A$. Thus one may remove the small $\beta_i$ term from \cite[Lemma A.7]{GTarith} in the $A$-irrational setting; in fact this version can also be deduced from \cite[Lemma A.7]{GTarith} itself by a compactness argument.
By the same procedure referred to above, one can then obtain a version of Proposition \ref{prop:full_fact} for arbitrary polynomials (with no periodicity assumption), where $\gamma$ is periodic instead of $\Gamma$-valued, with period bounded in terms of $\mathcal{F}(M)$. Thus one may factorize an arbitrary polynomial $g(n)$ into
\begin{equation*}
	\text{constant} \quad \times \quad \text{highly $A$-irrational} \quad \times \quad \text{boundedly periodic},
\end{equation*}
doing away with the `smooth' part of the factorization in \cite[Lemma 2.10]{GTarith}. The downside of this (slight) simplification is that one has only $A$-irrationality and not $(A,N)$-irrationality, though this is catered for in the periodic setting by the corresponding counting lemma, Theorem \ref{thm:periodic_counting}. Let us also note that one can deduce a version of Proposition \ref{prop:full_fact} from the above factorization, by dilating the variable $n$ by some fixed integer so that the periodic part becomes $\Gamma$-valued. However, one needs to ensure that this modification conserves irrationality, and the argument ends up being somewhat less clean than the one presented above.
\end{remark}

\begin{remark}
We are mainly interested in periodic polynomials in this paper, but it is interesting to note the wider applicability of Proposition \ref{prop:full_fact}. Indeed, subject to the normalizing condition $g(0)=\id_G$, the assumption $g(q\Z)\subset \Gamma$ is strictly weaker than $q$-periodicity mod $\Gamma$: consider for example a sequence $g^n h^n$ where $g,h$ are $q$th roots mod $\Gamma$. Furthermore, among the polynomials $g$ satisfying $g(0)=\id_G$, those with $g(q\Z)\subset \Gamma$ form a subgroup of $\poly(\Z,G_\bullet)$, whereas those that are $q$-periodic mod $\Gamma$ do not. Note also that the factorization theorem as stated above applies to finite products of periodic polynomials with trivial constant term, even if they have \emph{different} periods.
\end{remark}

\section{Constructing a periodic, irrational polynomial}\label{section:construction}
Thanks to the regularity and counting lemmas, understanding a discrete average across some system of linear forms $\Lf$ is essentially reduced to considering integrals of the form $\int_{G^\Lf/\Gamma^\Lf} F$ for Lipschitz functions $F$ and bounded-complexity nilmanifolds $G/\Gamma$. We now work in the converse direction: given such an integral, and some large period $q$, we want to approximate the integral by a discrete average involving some appropriate $q$-periodic orbit. More precisely, we want to find a sequence $g\in\poly(\Z,G_\bullet)$ that is $q$-periodic mod $\Gamma$ and has its orbits $g^\Lf$ equidistributed in $G^\Lf/\Gamma^\Lf$. To this end, in view of the counting lemma, we shall find a highly irrational $g$.

\begin{propositionbk}[Existence of a periodic, irrational polynomial]\label{prop:periodic_irr}
Let $(G/\Gamma,G_\bullet,\X)$ be a filtered nilmanifold of degree at most $s$ and dimension $m$. Then for any integer $q \geq (2A)^m$ with $p_1(q)\geq A$ there exists $g \in \poly(\Z,G_\bullet)$ that is $q$-periodic mod $\Gamma$ and $A$-irrational.
\end{propositionbk}

Our proof of this proposition occupies the remainder of this section. The main difficulty behind the result is that $q$-periodicity is in general not straightforwardly characterized in terms of Taylor coefficients, the objects central to the notion of irrationality. However, there are some instructive cases in which $q$-periodicity is easily related to these coefficients. For instance, if $g(n) = g_1^n$ is linear, then $q$-periodicity simply corresponds to $g_1$ being a $q$th root mod $\Gamma$, and it is not hard to construct an irrational $q$th root. For general filtrations, however, it is impossible for linear polynomials to be irrational, since for these polynomials any Taylor coefficient $g_i$ with $i \geq 2$ is trivial. Another case is when the group is abelian; for example a polynomial $g(n) = a_0+ a_1 n + a_2 \binom{n}{2} + \cdots + a_s \binom{n}{s}$ over $\R$ is $q$-periodic mod $\Z$ for a prime $q > s$ if and only if $a_i$ is a $q$th root mod $\Z$ for each $i\geq 1$, i.e. $a_i \in \Z/q$.\footnote{Also, as explained in \cite[Appendix A]{GTarith}, irrationality in this example consists in $a_s$ not being a rational with small denominator.} In general, however, each Taylor coefficient $g_i$ being a $q$th root is not sufficient for $g$ to be $q$-periodic; it is not hard to construct examples of this using real $4\times 4$ upper-unitriangular matrices.

In both cases above, what yields the simple characterization of periodicity is the ambient commutativity, which fails in the general setting. Taking heed of this, our proof builds up the desired polynomial sequence iteratively, starting in the degree 1 setting of $G_1/G_2$, and working at stage $i$ essentially with $G_i/G_{i+1}$, thus benefiting from commutativity at various points of the construction.

The irrationality input will come from the following lemma. (Recall the notation $r_i = m_i-m_{i+1}$.)

\begin{lemma}\label{lemma:irr_coset_qth_root}
	Let $(G/\Gamma, G_\bullet, \X)$ be a nilmanifold of degree at most $s$ and let $i \in [s]$. Let $A > 0$ and let $q \geq (2A)^{r_i}$ be an integer with $p_1(q)\geq A$. Then, for any $h \in G_i$, there exists $w \in G_i$ that is a $q$th root mod $\Gamma_i$ such that the product $h w$ is $A$-irrational \textup{(}in $G_i$\textup{)}.
\end{lemma}
\begin{proof}
For any $\gamma \in \Gamma_i$, there is a $w \in G_i$ for which $w^q = \gamma$, namely $w = \exp(\frac{1}{q}\log \gamma)$. We shall thus focus on picking $\gamma$ instead of $w$. Now, $h w$ is $A$-irrational if and only if $\xi(h w) \notin \Z$ for any non-trivial $i$th-level character $\xi$ of complexity at most $A$, that is iff $\xi(h^q \gamma) \notin q\Z$, which will certainly be the case, by Lemma \ref{lemma:frequency_vector}, if
	\[ k \cdot \psi_i(\gamma) \neq a_k \mod q \quad \text{for any $k \in \Z^r$ with $0 < \abs{k} \leq A$,} \]
where the $a_k$ are some real numbers coming from the $\xi(h^q)$, and $r := r_i$. All we need to do, then, is pick an integer vector $t \in \Z^r$ such that $k \cdot t \neq a_k \mod q$ for any such $k$, after which we simply set $\gamma := \psi_i^{-1}(t)$. But we can do this by a simple counting argument: for any $k \in \Z^r$ with $\hcf(k,q) = 1$, there are precisely $q^{r-1}$ solutions $t \in [q]^r$ to $k \cdot t = a_k \mod q$. Since there are at most $(A+1)^r-1$ vectors $k$ to be considered, provided $q^r \geq (A+1)^r q^{r-1}$ and $q$ only has prime factors bigger than $A$, there will be some $t \in [q]^r$ such that $k \cdot t \neq a_k \mod q$ for any $k$ with $0 < \abs{k} \leq A$.
\end{proof}
\begin{remark}
	The element $h w$ produced above is actually irrational in a stronger sense than the lemma suggests: it satisfies $\xi(h w) \notin \Z$ even if $\xi$ is not required to vanish on the groups $[G_j, G_{i-j}]$ (as $i$th-level characters in general are).
\end{remark}

We shall also require the following lemma on the Taylor coefficients of a polynomial with restricted derivatives.

\begin{lemma}[Taylor coefficients of differentiated polynomials]\label{lemma:shifted_Taylor}
Let $g \in \poly(\Z, G_\bullet)$ where $G_\bullet$ has degree at most $s$. If $\partial_{h_i} \cdots \partial_{h_1} g(n) \in G_{i+1}$ for all $i \geq 0$ and $h_1,\ldots,h_i, n \in \Z$, then we have $g_i \in G_{i+1}$ for each Taylor coefficient $g_i$ of $g$.
\end{lemma}
\begin{proof}	
This follows essentially from the fact that Lemma \ref{lemma:Taylor} remains valid under the weaker assumption that $G_\bullet$ is a \emph{prefiltration}\footnote{Following \cite{GTOrb}, a prefiltration is like a filtration but with the weaker requirement $G \supset G_0 \supset G_1$.} rather than a filtration, as recorded in Appendix \ref{appendix:Taylor}.
Indeed, the assumption on $g$ ensures that it lies in $\poly(\Z, G_\bullet^{+1})$, where $G_\bullet^{+1}$ is the prefiltration $(G_{i+1})_{i\geq 0}$ of degree at most $s-1$. Thus we can write $g(n) = g_0 g_1^n \cdots g_{s-1}^{\binom{n}{s-1}}$ for some $g_i \in G_{i+1}$ by Lemma \ref{lemma:prefiltration_Taylor}. The result now follows by the uniqueness of Taylor coefficients.
\end{proof}

We can now prove the main result of this section, following essentially the above-mentioned iterative process.

\begin{proof}[Proof of Proposition \ref{prop:periodic_irr}]
At stage $i$ of the proof we obtain a polynomial $g \in \poly(\Z,G_\bullet)$ such that $g(n+q)^{-1} g(n) \in \Gamma \cdot G_{i+1}$ for all $n \in \Z$ and such that the first $i+1$ Taylor coefficients of $g$ are $A$-irrational. We shall then be done after stage $s$.
	
For $i=0$ we set $g(n) = \id_G$ for all $n$; this trivially satisfies the required properties since there are no non-trivial $0$th level characters. Suppose, then, that we have a polynomial $g \in \poly(\Z, G_\bullet)$ such that $n \mapsto g(n+q)^{-1}g(n)$ takes values in $\Gamma\cdot G_i$ and such that $g_0,g_1,\ldots,g_{i-1}$ are $A$-irrational; we shall use this to produce a new polynomial $\tilde g$ with $\tilde g(n+q)^{-1}\tilde g(n)\in \Gamma\cdot G_{i+1}$ for all $n$ and with $\tilde g_j$ being $A$-irrational for $j=0,1,\ldots, i$. We may suppose that $g(n) = g_1^n \cdots g_{i-1}^{\binom{n}{i-1}}$. Set $h(n) = g(n+q)^{-1} g(n)$; this is a polynomial map, and by assumption it is $\Gamma \cdot G_i$-valued. By Lemma \ref{lemma:Taylor} we can therefore write $h(n)G_i = \gamma_0 \gamma_1^n \cdots \gamma_{i-1}^{\binom{n}{i-1}} G_i$ for some $\gamma_j \in \Gamma_j$, which we take to be the Taylor coefficients of a polynomial $\gamma \in \poly(\Z, G_\bullet)$. Thus we may factorize 
\begin{equation}
	h(n) = \gamma(n) \tilde{h}(n), \label{eqn:h_fact}
\end{equation}
where $\tilde{h} \in \poly(\Z, G_\bullet)$ is $G_i$-valued. We shall attempt to cancel out the contribution of this $G_i$-valued part $\tilde{h}$, and for this we need some information about its Taylor coefficients. First we have $\tilde{h}_j \in G_i$ for all $j$, by Lemma \ref{lemma:Taylor}. Then, looking at \eqref{eqn:h_fact} mod $G_{i+1}$ and using the centrality of $G_i$ mod $G_{i+1}$ we have
\[ h_0 h_1^n \cdots h_i^{\binom{n}{i}} G_{i+1} = \gamma_0 \tilde{h}_0 (\gamma_1 \tilde{h}_1)^n \cdots (\gamma_{i-1} \tilde{h}_{i-1})^{\binom{n}{i-1}} \tilde{h}_i^{\binom{n}{i}} G_{i+1}, \]
whence $h_i = \tilde{h}_i \mod G_{i+1}$. But $h$ is a `differentiated' polynomial, $h = \partial_q g^{-1}$, and so $\partial_{d_i} \cdots \partial_{d_1} h(n) \in G_{i+1}$ for all $d_1, \ldots, d_i, n \in \Z$, whence $h_i$---and so $\tilde{h}_i$---is trivial mod $G_{i+1}$ by Lemma \ref{lemma:shifted_Taylor}.

We are now almost ready to produce our new polynomial: it will be $\tilde{g}(n) := g(n) \ell(n) w^{\binom{n}{i}}$ for some $G_i$-valued polynomial $\ell \in \poly(\Z,G_\bullet)$ and some $q$th root $w \in G_i$ that we shall specify shortly. In fact we shall pick $\ell$ to be essentially an integral of (the inverse of) $\tilde{h}$: it will satisfy 
\begin{equation}\label{eqn:int_h} \ell(n)^{-1} \ell(n+q) G_{i+1} = \tilde{h}(n) G_{i+1}. \end{equation}
We can obtain such an $\ell$ by picking its Taylor coefficients $\ell_1, \ldots, \ell_i$ inductively to satisfy the system
\begin{eqnarray*}
	\ell_i^q &=& \tilde{h}_{i-1} \\
	\ell_{i-1}^q \ell_i^{\binom{q}{2}} &=& \tilde{h}_{i-2} \\
					 &\vdots& \\
	\ell_1^q \ell_2^{\binom{q}{2}} \cdots \ell_i^{\binom{q}{i}} &=& \tilde{h}_0.
\end{eqnarray*}
This yields coefficients $\ell_j \in G_i$ since each $\tilde{h}_j$ lies in $G_i$. Since $G_i$ is central mod $G_{i+1}$, \eqref{eqn:int_h} is easily seen to hold using the identity $\binom{n+q}{j} - \binom{n}{j} = \sum_{k=0}^{j-1} \binom{q}{j-k} \binom{n}{k}$. We then have
\[ \tilde{g}(n+q)^{-1} \tilde{g}(n) G_{i+1} = \gamma(n) \tilde{h}(n) \ell(n+q)^{-1} \ell(n) w^{\binom{n}{i}-\binom{n+q}{i}} G_{i+1} \in w^{\binom{n}{i}-\binom{n+q}{i}} \Gamma \cdot G_{i+1}. \]
Thus $\tilde{g}(n+q)^{-1} \tilde{g}(n)\in \Gamma\cdot G_{i+1}$ for all $n$ as desired provided $w \in G_i$ is a $q$th root mod $\Gamma_i$. But we also need $\tilde{g}_0,\ldots, \tilde{g}_i$ to be $A$-irrational. To this end, note that for each $j < i$ the Taylor coefficient $\tilde{g}_j$ is automatically $A$-irrational since it is congruent to $g_j$ mod $G_i$, and so we need only consider $\tilde{g}_i$. Now, mod $G_{i+1}$ we have $\tilde{g}_i = \ell_i w$; we thus pick $w \in G_i$ to be a $q$th root mod $\Gamma_i$ for which $\ell_i w$ is $A$-irrational, as we may by Lemma \ref{lemma:irr_coset_qth_root}, and the proof is complete.
\end{proof}

\section{Transference: moving from $\Zmod{N}$ to $\Zmod{M}$}\label{section:transference}

We now have all the tools required to prove Theorem \ref{pertranfn}, the result lying at the heart of the combinatorial applications in this paper. We shall in fact prove the following mild generalization.

\begin{theorembk}[Periodic transference\footnote{See also \S\ref{subs:Manners} for an update.}]\label{thm:periodic_transference}
Let $\cN\subset \N$ be a set of characteristic $0$, let $L \geq 1$ be an integer and let $\delta \in (0,1)$. Then for any $N \in \cN$ and $M \in \N$ with $p_1(N),p_1(M)\geq N_0(\delta,L,\cN)$, and any function $f : \Zmod{N} \to [0,1]$, there is a function $f' : \Zmod{M} \to [0,1]$ such that, for any system $\Lf$ of integer linear forms of size at most $L$, any two of which are linearly independent, we have
$\abs{ \Sol_\Lf(f' : \Zmod{M}) - \Sol_\Lf(f : \Zmod{N}) } \leq \delta$.
\end{theorembk}

\begin{proof}
Let $\epsilon > 0$ and $\mathcal{F} : \R^+ \to \R^+$ be a growth function, both to be specified in terms of $\delta$ and $L$ later, and let $s = s(L)$ be as given by Theorem \ref{thm:GvNdiag}. Assuming $p_1(N)>O_{s,\epsilon,\mathcal{F},\cN}(1)$, we apply Theorem \ref{thm:periodic_regularity} to $f$ to obtain a decomposition $f = f_\nil + f_\sml + f_\unf$ and an integer $Q = O_{\epsilon,L,\mathcal{F},\cN}(1)$ such that
\begin{enumerate}
\item $f_\nil = F(g(n)\Gamma)$, where $(g(n)\Gamma)$ is an $N$-periodic polynomial orbit on some nilmanifold $(G/\Gamma, G_\bullet, \X)$ of degree at most $s$ and complexity at most $Q$, $g \in \poly(\Z,G_\bullet)$ is $\mathcal{F}(Q)$-irrational, and $F : G/\Gamma \to \C$ satisfies $\norm{F}_{\Lip(\X)} \leq Q$,
\item $\norm{f_\sml}_2 \leq \epsilon$,
\item $\norm{f_\unf}_{U^{s+1}} \leq 1/\mathcal{F}(Q)$, and
\item $f_\nil$ and $f_\nil + f_\sml$ take values in $[0,1]$.
\end{enumerate}
Furthermore, since $f_\nil$ is $[0,1]$-valued, we may assume that $F$ is real-valued by taking real parts, and then by replacing it with $\max(\min(F,1),0)$ we may in fact assume that it is also $[0,1]$-valued; neither of these operations can increase $\norm{F}_{\Lip(\X)}$.

Now let $\Lf = (\lf_1, \ldots, \lf_t)$ be any system of pairwise independent linear forms $\lf_i : \Z^D \to \Z$ of size at most $L$. Then by Theorem \ref{thm:GvNdiag} and \eqref{eqn:sol_L1} we have
\[ \abs{ \Sol_\Lf(f) - \Sol_\Lf(f_\nil) } \leq L \epsilon + L/\mathcal{F}(Q). \]
The parameter $\epsilon$ will play no further role, so let us fix already $\epsilon = \delta/3L$.\\
\indent We now deal with $\Sol_\Lf(f_\nil)$ using the periodic counting lemma, Theorem \ref{thm:periodic_counting}. The Lipschitz function we use is $F^{\otimes t} : G^t/\Gamma^t \to \C$, given by $F^{\otimes t}(x_1, \ldots, x_t) = F(x_1)\cdots F(x_t)$; note that this has $\norm{F^{\otimes t}}_{\Lip(\X^t)} \leq C_{Q,L}$ by Lemma \ref{lemma:tensor_Lip}. Applying the counting lemma to this function, with parameter $Q' := \max(L,s(L),Q, C_{Q,L})$ instead of $Q$, we obtain
\[ \Sol_\Lf(f_\nil) = \E_{\mathbf{n}\in \Zmod{N}^D} F^{\otimes t}(g^{\Lf}(\mathbf{n})\Gamma^t) = \int_{g(0)^\Delta\, G^\Lf/\Gamma^\Lf} F^{\otimes t} + o_{\mathcal{F}(Q) \to \infty;Q,L}(1). \]
We can choose $\mathcal{F}$ with sufficiently fast growth in terms of $\delta$ and $L$ to then have
\[ \abs{ \Sol_\Lf(f) - \int_{g(0)^\Delta\,G^\Lf/\Gamma^\Lf} F^{\otimes t} } \leq 2\delta/3. \]

We now transfer to the group $\Zmod{M}$. Let $A=A(Q')$ be large enough so that the error term in Theorem \ref{thm:periodic_counting} is at most $\delta/3$, and let $h \in \poly(\Z,G_\bullet)$ be the $M$-periodic, $A$-irrational polynomial given by Proposition \ref{prop:periodic_irr}, noting that we may assume $h(0)= g(0)$; provided $p_1(M)$ is large enough we can find this. Theorem \ref{thm:periodic_counting} then gives us
\[ \abs{ \int_{g(0)^\Delta\,G^\Lf/\Gamma^\Lf} F^{\otimes t} - \E_{\mathbf{n}\in \Zmod{M}^D} F^{\otimes t}(h^{\Lf}(\mathbf{n})\Gamma^t) } \leq \delta/3. \]
Setting $f' : \Zmod{M} \to [0,1]$, $f'(n) = F(h(n)\Gamma)$, the expectation above is precisely $\Sol_\Lf(f')$, so that
$ \abs{ \Sol_\Lf(f') - \Sol_\Lf(f) } \leq \delta $ as required.
\end{proof}

\section{Applications}\label{section:applications}

We are now ready to establish Theorems \ref{mincount} and \ref{maxdensity}. We begin with the former, which we restate now in the stronger form that was mentioned in the introduction. In the version below, the primality restriction on $N$ is replaced with the weaker requirement that $N$ belong to a set of characteristic 0 (recall Definition \ref{defn:char0}).
\enlargethispage{\baselineskip}
\begingroup
\makeatletter
\apptocmd{\thetheorem}{\unless\ifx\protect\@unexpandable@protect\protect\footnote{See \S\ref{subs:Manners} for an update providing a more general version of this theorem.}\fi}{}{}
\makeatother
\begin{theorem}\label{mincount*}
Let $\Lf$ be a system of integer linear forms, any two of which are linearly independent. Then for any $\alpha \in [0,1]$ there is a number $m_\Lf(\alpha)$ such that the sequence $m_\Lf(\alpha, N)\to m_\Lf(\alpha)$ as $N \to \infty$ through $\mathcal{N}$,
for any set $\mathcal{N}\subset \N$ of characteristic $0$.
\end{theorem}
\endgroup
The proof involves the following standard result.
\begin{lemma}\label{MainF2S}
For any positive integer $d$ and any function $f:\Zmod{N}\to [0,1]$ there exists a set $A\subset \Zmod{N}$ such that $\norm{1_A-f}_{U^d(\Zmod{N})}\ll_d N^{-1/2^d}$.
\end{lemma}
This follows easily from \cite[Exercise 11.1.17]{T-V}; one can prove it by picking the set $A$ randomly, letting each element $x \in \Zmod{N}$ lie in $A$ with probability $f(x)$ independently.

\begin{proof}[Proof of Theorem \ref{mincount*}.]
Suppose $\Lf$ has size at most $L$, and fix any $\alpha\in [0,1]$. We first establish convergence for any fixed $\mathcal{N}$ (to some limit possibly dependent on $\cN$).

Given such a set $\mathcal{N}$, fix an arbitrary $\epsilon\in(0,1)$. Let $N_0$ be the integer obtained by applying Theorem \ref{thm:periodic_transference} with $\delta=\epsilon/2(L+1)$, and let $s$ and $C_L$ be as given by Theorem \ref{thm:GvNdiag}. 
Let $N,M$ be any elements of $\mathcal{N}$ such that $N,M\gg_s (L/\delta)^{2^{s+1}}$ and $p_1(N),p_1(M)>\max\{L,C_L,N_0\}$, and let $A$ be a subset of $\Zmod{N}$ of size at least $\alpha N$ satisfying $\Sol_\Lf(A)=m_\Lf(\alpha,N)$. Then Theorem \ref{thm:periodic_transference} gives us a function $f' : \Zmod{M}\to [0,1]$ with $\E_{\Zmod{M}} f'\geq \alpha-\delta$ and such that $\Sol_\Lf(f') \leq m_\Lf(\alpha,N)+\delta$. Applying Lemma \ref{MainF2S} to $f'$ with $d=s+1$, we obtain a subset $B'$ of $\Zmod{M}$ such that $\norm{1_{B'}-f'}_{U^{s+1}(\Zmod{M})}\leq \delta/L$. We then have $\abs{B'}\geq (\alpha-2\delta)M$, since $\abs{\E_{\Zmod{M}} (f'-1_{B'})} \leq \norm{f'-1_{B'}}_{U^{s+1}(\Zmod{M})}$. Moreover, Theorem \ref{thm:GvNdiag} gives us that $\Sol_\Lf(B') \leq m_\Lf(\alpha,N)+2\delta$. Now we add at most $2\delta M$ elements from $\Zmod{M}\setminus B'$ to $B'$, obtaining a set $B\subset \Zmod{M}$ of size at least $\alpha M$ which, by \eqref{eqn:sol_L1}, satisfies $\Sol_\Lf(B)\leq \Sol_\Lf(B')+2\delta L$. It follows that $\Sol_\Lf(B)\leq m_\Lf(\alpha,N)+2(L+1)\delta$, whence $m_\Lf(\alpha,M)\leq m_\Lf(\alpha,N)+\epsilon$. Arguing the same way with $N$ and $M$ interchanged, we obtain $m_\Lf(\alpha,N)\leq m_\Lf(\alpha,M)+\epsilon$. Thus 
 $\big( m_\Lf(\alpha,N) \big)_{N\in\mathcal{N}}$ is a Cauchy sequence.

Now if $\mathcal{N}'$ is another set of characteristic $0$, then noting that $\mathcal{N}\cup \mathcal{N}'$ is also of characteristic 0, we deduce that the limit of $(m_\Lf(\alpha,N))$ for $\mathcal{N}'$ is equal to that for $\mathcal{N}$.
\end{proof}

We now turn to Theorem \ref{maxdensity}, which we shall establish in the following stronger form.

\begin{theorem}\label{maxdensity*}
Let $\mathcal{F}$ be a finite family of systems of integer linear forms, in each of which the forms are pairwise linearly independent. Then there is
a number $d_\mathcal{F}$ such that $d_\mathcal{F}(\Zmod{N})\to d_\mathcal{F}$ as $N\to\infty$ through $\mathcal{N}$, for any set $\mathcal{N}\subset \N$ of characteristic $0$.
\end{theorem}

We shall use the following result, known as an \emph{arithmetic removal lemma}, which follows from the more general removal lemma of Kr\'al, Serra, and Vena \cite{KSVGR}.

\begin{theorem}\label{removal}
Let $\Lf$ be a system of linear forms $\lf_1,\ldots,\lf_t:\Z^D\to \Z$. Then there exists a positive integer $K$ such that the following holds. For any $\epsilon>0$, there exists $\delta=\delta(\epsilon,\Lf)>0$ such that if $N\in \N$ is prime to $K$, and 
$A_1,\ldots, A_t$ are subsets of $\Zmod{N}$ such that $\Sol_\Lf(A_1,\ldots,A_t)\leq \delta$, then there exist sets $E_i\subset \Zmod{N}$ with $\abs{E_i}\leq \epsilon N$ for all $i\in [t]$, such that $\Sol_\Lf(A_1\setminus E_1,\ldots,A_t\setminus E_t)=0$.
\end{theorem}
The proof is a straightforward deduction, given in Appendix \ref{appendix:removal}.

\begin{proof}[Proof of Theorem \ref{maxdensity*}]
Let $n$ be the cardinality of the given finite family $\mathcal{F}$, and let $L$ be a uniform upper bound on the sizes of the systems in $\mathcal{F}$. As in the proof of Theorem \ref{mincount*}, it suffices to establish convergence for any given $\mathcal{N}$.

Having fixed $\mathcal{N}$, let us fix an arbitrary $\epsilon>0$. Let $\delta \in (0, \epsilon/2)$ be such that Theorem \ref{removal} holds for each $\Lf\in\mathcal{F}$, with initial parameter $\epsilon/2n$. Now let $C=C(\epsilon,L)$ be such that Theorem \ref{thm:periodic_transference} and Lemma \ref{MainF2S} hold with main upper bound $\delta/2L$, for any $N,M\in \mathcal{N}$ with $p_1(N),p_1(M) > C$. We claim that $\abs{ \md_\mathcal{F}(\Zmod{N}) - \md_\mathcal{F}(\Zmod{M}) } \leq \epsilon$ for any such $N,M$.

To see this let $\alpha= \md_\mathcal{F}(\Zmod{N})$ and let $A\subset \Zmod{N}$ be $\mathcal{F}$-free with $\abs{A}=\alpha N$. Then by Theorem \ref{thm:periodic_transference} there exists $f' : \Zmod{M}\to [0,1]$ with $\abs{\E_{\Zmod{M}} f'-\alpha}\leq \delta/2$ such that $\Sol_\Lf(f') \leq \delta/2$ for every $\Lf \in\mathcal{F}$. Applying Lemma \ref{MainF2S} to $f'$, we obtain a subset $B'$ of $\Zmod{M}$ with $\norm{1_{B'}-f'}_{U^{s+1}}\leq \delta/2L$, where $s=s(L)$ is as given by Theorem \ref{thm:GvNdiag}. As in the proof of Theorem \ref{mincount*}, this then implies $\abs{\alpha-\abs{B'}/M}\leq \delta$, and Theorem \ref{thm:GvNdiag} also gives that $\Sol_\Lf(B') \leq \delta$ for every $\Lf\in\mathcal{F}$.

Now, by our choice of $\delta$, Lemma \ref{removal} applied to each $\Lf\in\mathcal{F}$ gives us an $\mathcal{F}$-free subset $B$ of $B'$ with $\abs{\alpha-\abs{B}/M}\leq \delta+\epsilon/2$, whence $\md_\mathcal{F}(\Zmod{M})\geq \md_\mathcal{F}(\Zmod{N})-\epsilon$. Arguing similarly with $N,M$ interchanged, our claim follows. Thus $\big(\md_\mathcal{F}(\Zmod{N})\big)_{N\in\mathcal{N}}$ is a Cauchy sequence.
\end{proof}

We close this section with the following result mentioned in the introduction.
\begin{lemma}\label{lemma:non-invariant-free-sets}
Let $\mathcal{F}$ be a finite family of non-invariant systems of integer linear forms, in each of which the forms are pairwise independent. Then
$d_\mathcal{F}>0$.
\end{lemma}
\begin{proof}
After converting from systems of linear forms to systems of linear equations, this lemma follows immediately from a similar result for families of single equations. The latter result was recorded as \cite[Proposition 1.4]{candela-sisask} and its proof consisted in a simple construction based on an idea employed by Ruzsa \cite[Theorem 2.1]{Ruzsa}. To convert to systems of equations, then, assign to each $\Lf\in\mathcal{F}$ an integer matrix $\Lambda$ as in the proof of Theorem \ref{removal}. Note that the pairwise independence condition implies that any row of any such $\Lambda$ has at least three non-zero coefficients. Labeling these matrices $\Lambda_1,\Lambda_2,\ldots,\Lambda_n$, we now form a family 
of non-invariant integer linear forms $L_1,\ldots,L_n$, by defining the coefficients of $L_i$ to be the entries in a chosen row of $\Lambda_i$ not summing to 0. It is clear that $d_\mathcal{F}(\Zmod{N})$ is always at least the maximum density of a subset of $\Zmod{N}$ avoiding the equations $L_i(x) =0$, $i\in [n]$, so we are done.
\end{proof}

\section{Remarks}\label{section:remarks}

If we consider vector spaces over finite fields instead of cyclic groups then the analogue of Theorem \ref{pertranfn} is a more exact statement telling us that for any $n$ and $m\geq n$ we can transfer a function on $\mathbb{F}^n$ to one on $\mathbb{F}^m$ having equal solution-measures for any $\Lf$. In contrast with Theorem \ref{pertranfn}, however, the latter result is rather trivial due to the fact that $\F^n$ can be embedded as a subgroup of $\F^m$. This indicates that the finite-field viewpoint is less useful here than it is for several other well-known problems in additive combinatorics, the non-triviality of Theorem \ref{pertranfn} being more strongly related to the cyclic group setting, in which there can be a complete lack of non-trivial subgroups.

The questions of the convergence of the minimal solution measures and the analogues of $d_{\mathcal{F}}$ are also interesting for the integral setting of $[N]$ (once defined appropriately); indeed, the latter question was raised for single linear equations by Ruzsa \cite{Ruzsa}. One may establish a transference result in this setting (and thus convergence) relatively straightforwardly from Green and Tao's results in \cite{GTarith}: one can follow the structure of the proof of Theorem \ref{pertranfn}, except that to obtain $f' : [M] \to [0,1]$ (for $M \geq N$) one can simply extend the domain of definition of $f_\nil$, the main point being that since $f_\nil$ equidistributes well already up to time $N$, it automatically does so up to time $M$ as well. 

In the setting of Croot's original convergence result \cite{croot:3APminconv}, there is a nice relation between the minimum and maximum possible counts of 3-term progressions in sets of various densities, thanks to the formula
$\Sol_\threeAP(A) + \Sol_\threeAP(A^c) = 1 - 3\alpha + 3\alpha^2$
for sets $A \subset \Zmod{N}$ of density $\alpha$ (provided $N$ is odd). Thus a set has the minimal number of 3-term progressions for sets of density $\alpha$ if and only if its complement has the maximal number for sets of density $1-\alpha$. From this it is immediate that the quantity $M_\threeAP(\alpha, N) := \max_{ A \subset \Zmod{N},\, \abs{A} \leq \alpha N} \Sol_\threeAP(A)$ also converges as $N \to \infty$ over primes. There is a similar relation for solution counts of other single linear equations in an odd number of variables, but for more general systems of equations no such relation holds. Nevertheless, the methods of this paper do of course allow one to deduce convergence in this regime: 

\begin{theorem}
Let $\Lf$ be a system of integer linear forms, any two of which are linearly independent. For any $\alpha\in[0,1]$ and $N\in \N$, let $M_\Lf(\alpha, N) := \max_{A \subset \Zmod{N},\, \abs{A} \leq \alpha N} \Sol_\Lf(A)$.
Then $M_\Lf(\alpha, N)$ converges as $N \to \infty$ through primes.
\end{theorem}
\noindent These quantities are very natural from a combinatorial perspective, as they capture how structured a set of a given density can be. It would be interesting to know more about these limits in general.

It would also be interesting to identify `limit objects' on which to study quantities such as $m_\Lf(\alpha)$ and $d_\Lf$ directly, for a general system $\Lf$ of finite complexity. For systems $\Lf$ corresponding to a single linear equation, the circle group is already known to be a suitable limit object, in that $m_\Lf(\alpha) = m_\Lf(\alpha, \T)$ and $d_\Lf = d_\Lf(\T)$, where these quantities are defined naturally in terms of measurable subsets of $\T$ (see \cite{candela-sisask, olof-thesis}). For more general systems, one possibility would be to have, for each value of $s\in \N$, a single space $X_s$ on which $d_\Lf$ can be studied directly for any system $\Lf$ of complexity\footnote{This notion of complexity was defined in \cite{Gowers-Wolf:complexity}.} at most $s$ (in particular we would have $X_1=\T$). One may expect to characterize such a space $X_s$ in terms of nilmanifolds of degree at most $s$. 

Finally, let us note that one may obtain periodic analogues of the equidistribution results of \cite{GTOrb}, namely \cite[Theorems 1.16 \& 1.19]{GTOrb}, by similar considerations to those in this paper; we omit the details.

\subsection{Update}\label{subs:Manners}
Since the completion of this paper, Manners has released a preprint \cite{Manners} the main result of which affords a simplification in several results of this paper. The preprint concerns the inverse theorems for the Gowers norms, mentioned in Section \ref{section:inverse}. Of the known approaches to these theorems, in particular that of Green, Tao and Ziegler \cite{GTZ}, and that of Szegedy and Camarena-Szegedy \cite{Cam-Szeg, SzegedyHFA}, we adopted the latter's result (Theorem \ref{thm:perinverse}), as this ensured a particular periodicity property of the nilsequence in the theorem, a feature not present or needed in \cite{GTZ} but crucial for us. Manners showed, however, that one can deduce a periodic inverse theorem slightly stronger than Theorem \ref{thm:perinverse} directly from the inverse theorem of \cite{GTZ}. The strengthening, \cite[Theorem 1.4]{Manners}, consists in removing the need to fix a characteristic-$0$ family, working instead for all integers $N$. This obviates the need for considering such families in our results, meaning that the only condition on the orders of the cyclic groups is concerned directly with their smallest prime factors. Thus, for example, in Theorem \ref{thm:periodic_regularity}---the periodic regularity lemma with irrational nilsequence---one can simply set $\mathcal{N} = \N$, effectively removing all mention of $\mathcal{N}$, and similarly in Theorem \ref{thm:periodic_transference}. This in turn leads to the following version of our convergence result Theorem \ref{mincount*}.

\begin{theorem}\label{mincount**}
Let $\Lf$ be a system of integer linear forms, any two of which are linearly independent, and let $\alpha \in [0,1]$. Then $m_\phi(\alpha, N)$ converges as $p_1(N)\to \infty$. 
\end{theorem}
\newpage

\appendix

\section{Mal'cev bases}\label{appendix:Malcev}

This appendix gathers some technical tools on Mal'cev bases and related notions.

\begin{defn}[Mal'cev basis]
Let $(G/\Gamma, G_\bullet)$ be an $m$-dimensional filtered nilmanifold. A basis $\X = \{X_1,\ldots,X_m\}$ for the Lie algebra $\mathfrak{g}$ over $\R$ is called a \emph{Mal'cev basis} for $G/\Gamma$ adapted to $G_\bullet$ if the following conditions are satisfied.
\begin{enumerate}
\item For each $j \in [0, m-1]$ the subspace $\mathfrak{h}_j := \Span(X_{j+1},\ldots, X_m)$ is an ideal in $\mathfrak{g}$, and hence
$H_j := \exp \mathfrak{h}_j$ is a normal Lie subgroup of $G$.
\item For every $i\in [0,s]$ we have $G_i = H_{m-m_i}$.
\item \label{coords} Each $g \in G$ can be written uniquely as $\exp( t_1 X_1) \exp(t_2 X_2) \cdots \exp( t_m X_m)$, for
$t_i \in \R$.
\item $\Gamma=\{ \exp( t_1 X_1) \exp(t_2 X_2) \cdots \exp( t_m X_m) : t_i\in \Z\}$.
\end{enumerate}
\end{defn}

The Mal'cev coordinate map $\psi:G\to \R^m$ referred to in Section \ref{section:background} is then just the map sending $g \in G$ to its corresponding tuple $(t_1,\ldots,t_m) \in \R^m$ from \eqref{coords} above. 

Given a basis $\X$ on $G/\Gamma$, the following result describes a natural Mal'cev basis for a power $G^t/\Gamma^t$. This will enable us to relate Lipschitz norms on these two manifolds as needed in the proof of Theorem \ref{thm:periodic_transference}. 

\begin{lemma}\label{Xt}
Let $(G/\Gamma,G_\bullet)$ be an $m$-dimensional filtered nilmanifold of degree at most $s$ with a $Q$-rational Mal'cev basis $\X=\{X_1,\ldots, X_m\}$, and let $t\in \N$. Let $\mathfrak{g}^t$ denote the direct sum of $t$ copies of $\mathfrak{g}$, let $\X^t\subset \mathfrak{g}^t$ be the set of vectors $X_{i,j}=(0,\ldots,0,X_j,0,\ldots,0)$ where $X_j$ appears at the $i$th entry, and let $\X^t$ be ordered according to the colex order on $[t] \times [m]$. Then $\X^t$ is a $Q$-rational Mal'cev basis for $(G^t/\Gamma^t,G_\bullet^t)$.
\end{lemma}
The proof is routine, being based on the fact that the exponential map from $\mathfrak{g}^t$ to $G^t$ is given by $(v_1,...,v_m)\mapsto (\exp(v_1),\ldots,\exp(v_m))$ \cite{Vara}.

We now use the form of this basis to relate the metrics on $G/\Gamma$ and $G^t/\Gamma^t$. Recall that, given a Mal'cev basis $\X$ for a nilmanifold $(G/\Gamma, G_\bullet)$, the metric $d_G = d_{G, \X}$ was defined as the largest metric $d$ for which $d(x,y) \leq \norm{\psi(xy^{-1})}_{\infty}$ for all $x,y \in G$.\footnote{That is, $d(x,y) := \sup_{i \in I} d_i(x,y)$ where the $d_i$ are all the metrics satisfying the condition, whence $d_i(x,y) \leq d(x,y)$ for all such metrics.}

\begin{lemma}\label{lemma:d_vs_dt}
Let $\X$ be a Mal'cev basis for $(G/\Gamma, G_\bullet)$ and let $\X^t$ be the corresponding basis for $(G^t/\Gamma^t, G_\bullet^t)$ given by Lemma \ref{Xt}. Then, for any $x,y \in G^t$,
\[ d_{G^t}(x,y) \geq \max_{i \in [t]} d_{G}(x_i, y_i)\qquad\text{and}\qquad
d_{G^t/\Gamma^t}(x\Gamma^t,y\Gamma^t) \geq \max_{i \in [t]} d_{G/\Gamma}(x_i\Gamma, y_i\Gamma). \]
\end{lemma}
In fact it is not hard to show (using \cite[Lemma A.4]{GTOrb}) that the metrics $d_{G^t/\Gamma^t}(x\Gamma^t,y\Gamma^t)$ and $\max_{i \in [t]} d_{G/\Gamma}(x_i\Gamma, y_i\Gamma)$ on $G^t/\Gamma^t$ are Lipschitz equivalent with constant depending on the rationality bound for $\X$, but we do not need this fact here.
\begin{proof}
Write $\psi : G \to \R^m$ for the Mal'cev map on $G$ corresponding to $\X$, and $\psi_t : G^t \to \R^{m \times t}$ for the one on $G^t$ corresponding to $\X^t$. Thought of as a matrix, it is easy to see that 
\[ \psi_t(x_1, \ldots, x_t) = \big( \psi(x_1)^\intercal \ldots \psi(x_t)^\intercal \big). \]
From this it is immediate that the metric $d'(x,y) := \max_{i \in [t]} d_G(x_i, y_i)$ on $G^t$ satisfies $d'(x,y) \leq \norm{\psi_t(xy^{-1})}_\infty$ for all $x,y \in G^t$. Since $d_{G^t}$ is the largest metric satisfying this condition, we have $d'(x,y) \leq d_{G^t}(x,y)$ for all $x,y \in G^t$, which was the first claim.

The claim for $G^t/\Gamma^t$ then follows immediately since, for any $i \in [t]$,
\[ d_{G^t/\Gamma^t}(x\Gamma^t,y\Gamma^t) = \inf_{\gamma \in \Gamma^t} d_{G^t}(x, y\gamma) \geq \inf_{\gamma \in \Gamma^t} d(x_i, y_i\gamma_i) = d_{G/\Gamma}(x_i,y_i). \qedhere \]
\end{proof}

\noindent Finally we relate this to the corresponding Lipschitz norms. Recall that 
\[ \norm{F}_{\Lip(\X)} = \norm{F}_\infty + \sup_{x\Gamma \neq y\Gamma \in G/\Gamma} \frac{ \abs{ F(x\Gamma) - F(y\Gamma) } }{ d_{G/\Gamma}(x\Gamma, y\Gamma) }. \]

\begin{lemma}\label{lemma:tensor_Lip}
Let $(G/\Gamma, G_\bullet, \X)$ be a filtered nilmanifold, let $t$ be a positive integer, and let $F : G/\Gamma \to \C$ be a function. Then, writing $F^{\otimes t}$ for the function $(x_1, \ldots, x_t) \mapsto F(x_1) \cdots F(x_t)$ on the power $(G^t/\Gamma^t, G_\bullet^t, \X^t)$, we have $\norm{F^{\otimes t}}_{\Lip(\X^t)} \leq t \norm{F}_{\Lip(\X)}^t$.
\end{lemma}
\begin{proof}
By telescoping it is easy to see that
\[ \abs{F(x_1) \cdots F(x_t) - F(y_1) \cdots F(y_t) } \leq \norm{F}_\infty^{t-1} \sum_{i \in [t]} \abs{F(x_i) - F(y_i)}. \]
Dividing by $d_{G^t/\Gamma^t}(x,y)$ and applying Lemma \ref{lemma:d_vs_dt}, the claim follows.
\end{proof}

\section{Factorizing non-irrational polynomials}\label{appendix:factorization}

In this appendix we show how to deduce the full factorization result, Proposition \ref{prop:full_fact}, from the analogous result for Taylor coefficients, Lemma \ref{lemma:Taylor_fact}. This is completely similar to the deduction of \cite[Lemma 2.10]{GTarith} from \cite[Lemma A.7]{GTarith}, but for completeness we include the proof. The first step is to establish the following basic factorization, analogous to \cite[Lemma 2.9]{GTarith}.

\begin{lemma}[Basic factorization]\label{our_2.9_2}
Let $(G/\Gamma,G_\bullet, \X)$ be a nilmanifold of complexity at most $M$, and let $g\in \poly(\Z,G_\bullet)$ be such that $g(0)=\id_G$ and $g(q\Z) \subset \Gamma$ for some integer $q$ with $p_1(q)>A$. Then at least one of the following statements holds.
\begin{enumerate}
\item $g$ is $A$-irrational in $(G/\Gamma,G_\bullet, \X)$.
\item There exists a factorization $g=g'\gamma$, with $g'\in \poly(\Z,G'_\bullet)$ such that $g'(n)\Gamma'$ takes values in a subnilmanifold $(G'/\Gamma',G'_\bullet, \X')$ of $G/\Gamma$ of strictly smaller total dimension and of complexity $O_{M,A}(1)$, $g'(0) = \id_G$, and $\gamma \in \poly(\Z,G_\bullet)$ is $\Gamma$-valued.
\end{enumerate}
\end{lemma}
Note that $g'$ also satisfies $g'(q\Z) \subset \Gamma'$, indeed $G'\ni g'(qn) = g(qn)\gamma(qn)^{-1} \in \Gamma$. Similarly, if $(g(n)\Gamma)$ is $q$-periodic, then so is $(g'(n)\Gamma')$.
\begin{proof}
Assume $g$ is not $A$-irrational, with Taylor expansion $g(n)=g_1^{\binom{n}{1}}g_2^{\binom{n}{2}}\cdots g_s^{\binom{n}{s}}$, where $g_i\in G_i$ for each $i$. Lemma \ref{lemma:Taylor_fact} then implies that, for some $i\in [s]$, we have $g_i=g_i'\gamma_i$, where $g_i'\in \ker \xi_i$ for some non-trivial $i$th-level character $\xi_i: G \to \R$ of complexity at most $A$, and $\gamma_i \in \Gamma_i$. As in \cite{GTarith}, we shall now consider the cases $i>1$ and $i=1$ separately.

For $i>1$ we write $g(n)=g_{<i}(n)\;(g_i'\gamma_i)^{\binom{n}{i}}\;g_{>i}(n)$, where $g_{<i}(n)=g_0g_1^{\binom{n}{1}}\cdots g_{i-1}^{\binom{n}{i-1}}$ and $g_{>i}(n)=g_{i+1}^{\binom{n}{i+1}}\cdots g_s^{\binom{n}{s}}$. Now by \cite[(C.1)]{GTarith} we have
\[
(g_i'\gamma_i)^{\binom{n}{i}}={g_i'}^{\binom{n}{i}}\gamma_i^{\binom{n}{i}}\prod_\alpha g_\alpha^{Q_\alpha(\binom{n}{i})}
\]
where each $g_\alpha$ is an iterated commutator of $k_1=k_{1,\alpha}$ copies of $g_i'$ and $k_2=k_{2,\alpha}$ copies of $\gamma_i$, where $k_1,k_2\geq 1$ and $k_1+k_2\geq 2$, and where $Q_\alpha$ are polynomials of degree $\leq k_1+k_2$ with no constant term. It follows from this and the group property that $\prod_\alpha g_\alpha^{Q_\alpha\left(\binom{n}{i}\right)}$ is a $G_{2i}$-valued polynomial sequence. Therefore the sequence $\hat{g}_{>i}(n) := \prod_\alpha g_\alpha^{Q_\alpha\left(\binom{n}{i}\right)}g_{>i}(n)$ is a $G_{i+1}$-valued polynomial sequence. We can then write
\[
g(n)=g'(n)\gamma_i^{\binom{n}{i}},\qquad\text{where}\qquad g'(n)=g_{<i}(n){g_i'}^{\binom{n}{i}}\tilde{g}_{>i}(n),
\]
and $\tilde{g}_{>i}(n)=[{\gamma_i^{\binom{n}{i}}},\hat{g}_{>i}(n)] \hat{g}_{>i}(n)$ is a $G_{i+1}$-valued polynomial sequence. As in \cite{GTarith}, then, we have that $g'$ is in $\poly(\Z,G'_\bullet)$, where $G'/\Gamma'=G/\Gamma$ and $G'_j=G_j$ for $j\neq i$, with $G_i'= \ker(\xi_i)$. Indeed, $G'_\bullet$ is a filtration since $G_{i+1}, [G_j,G_{i-j}]\subset \ker(\xi_i)$ by definition, and $g'$ is a polynomial adapted to $G'_\bullet$ since $g_{<i}(n){g_i'}^{\binom{n}{i}}$ is (by Lemma \ref{lemma:Taylor}) and $\tilde{g}_{>i}$ is as well. This completes the case $i>1$, since $\dim(\ker(\xi_i)) < \dim(G_i)$.

For $i=1$ we can just set $G_0'=G_1'=\ker(\xi_1)$ as the first two terms of our new filtration, since $g_0=\id_G$, and our factorization is then $g(n)=g'(n)\gamma_1^n$, where $g'(n)=g_1'^n\tilde g_{>1}(n)$, with $\tilde g_{>1}$ taking values in $G_2'$.

In both cases above, to obtain an appropriate Mal'cev basis for the new filtration, note that we can simply apply \cite[Proposition A.10]{GTOrb}, $\ker \xi_i$ being boundedly rational since $\xi_i$ has bounded complexity.
\end{proof}
Recall that the complexity $M_0$ of a filtered nilmanifold is a common upper bound for the dimension, the degree, and the rationality of the Mal'cev basis. In particular the total dimension $\sum_i \dim(G_i)$ is bounded by $(s+1)M_0 \leq (M_0+1)M_0$. Recall also from Section \ref{section:nilmanifolds} the definition of the ``fractional part'' $\{g\}$ of $g\in G$ relative to $\Gamma$.

\begin{proof}[Proof of Proposition \ref{prop:full_fact}]
	If the given sequence $g$ is $\mathcal{F}(M_0)$-irrational in $(G/\Gamma, G_\bullet)$ then we are done; if not, then by Lemma \ref{our_2.9_2} we have $g=g_1\gamma_1$ where $\gamma_1$ is $\Gamma$-valued, $g_1$ is $G'$-valued, where $(G'/\Gamma',G'_\bullet)$ is a subnilmanifold of $(G/\Gamma,G_\bullet)$ of complexity $M_1 = O_{\mathcal{F}(M_0)}(1)$ and strictly smaller total dimension than $(G/\Gamma,G_\bullet)$, and moreover $g_1(0) = \id_G$ and $g_1(q\Z) \subset \Gamma'$. Now, if $g_1$ is $\mathcal{F}(M_1)$-irrational in $(G'/\Gamma',G'_\bullet)$, then we are done; otherwise we apply Lemma \ref{our_2.9_2} again to $g_1$.	
	Carrying on this way, the process must stop after at most $O_{M_0}(1)$ applications of Lemma \ref{our_2.9_2}, by the initial bound on the total dimension of $(G/\Gamma,G_\bullet)$, and the full factorization follows.
(Note that to be able to apply the lemma enough times, we need $p_1(q)$ greater than the final irrationality requirement we may end up with, which is $\mathcal{F}(M_j)=O_{\mathcal{F},M_0}(1)$ for some $M_j$ as constructed above.)\\
\indent For the final claim in the lemma, note that if $g$ is $q$-periodic mod $\Gamma$ then the first part of the lemma applied to the sequence $\{ g(0) \}^{-1}\, g \, [g(0)]^{-1}$ yields the claimed factorization $g = \varepsilon\, g'\, \gamma$.
\end{proof}

\section{Polynomials with respect to prefiltrations}\label{appendix:Taylor}

In this section we record some facts about polynomials with respect to prefiltrations. By a \emph{prefiltration} of degree at most $s$ in a group $G$ we here mean a sequence $(G_i)$ of subgroups of $G$ with $G \supset G_0 \supset G_1 \supset \cdots \supset G_s \supset G_{s+1} = \{ \id_G \}$ and $[G_i, G_j] \subset G_{i+j}$ for all $i,j \geq 0$. The definition of $\poly(\Z, G_\bullet)$ extends to prefiltrations with no change, consisting of all the maps $g : \Z \to G$ such that $\partial_{h_i} \cdots \partial_{h_1} g(n) \in G_i$ for all $i \geq 0$ and $h_1,\ldots, h_i, n \in \Z$. Moreover, as with filtrations, this space forms a group, as follows immediately from \cite[Proposition 6.5]{GTOrb}. We also have the following version of Lemma \ref{lemma:Taylor}.

\begin{lemma}[Taylor expansion for prefiltrations]\label{lemma:prefiltration_Taylor}
Let $g \in \poly(\Z, G_\bullet)$, where $G_\bullet$ is a prefiltration of degree at most $s$. Then there are unique coefficients $g_i \in G_i$ such that
\[ g(n) = g_0 g_1^n g_2^{\binom{n}{2}} \cdots g_s^{\binom{n}{s}}\quad \textrm{for all } n \in \Z.
\]
\end{lemma}
There are several ways to prove this; we follow a natural induction along the lines of Leibman \cite[\S 4.7]{LeibDiag} that makes use of the following lemma.

\begin{lemma}
Let $g,h \in \poly(\Z, G_\bullet)$ where $G_\bullet$ is a prefiltration of degree at most $s$. If $g(n)$ and $h(n)$ agree for $n = 0,1,\ldots,s$ then they agree for all $n$.
\end{lemma}
\begin{proof}
This follows immediately by induction by considering the polynomials $\partial_1 g, \partial_1 h$.
\end{proof}

\begin{proof}[Proof of Lemma \ref{lemma:prefiltration_Taylor}]
Let $g_0 := g(0) \in G_0$. Suppose $g_0, \ldots, g_i$ have been found so that $g_j \in G_j$, $g(n) = g_0 \cdots g_i^{\binom{n}{i}}$ for $n = 0,\ldots, i$, and $g(n) G_{i+1} = g_0 \cdots g_i^{\binom{n}{i}} G_{i+1}$ for all $n \in \Z$; note that this holds for $i = 0$ since $g \in \poly(\Z, G_\bullet)$. We then define
\[ g_{i+1} := \left( g_0 \cdots g_i^{\binom{i+1}{i}} \right)^{-1} g(i+1), \]
so that $g_{i+1} \in G_{i+1}$. Then certainly $g(n) = g_0 \cdots g_i^{\binom{n}{i}} g_{i+1}^{\binom{n}{i+1}}$ for $n = 0, 1, \ldots, i+1$. But the polynomials $g(n) G_{i+2}$ and $g_0 \cdots g_i^{\binom{n}{i}} g_{i+1}^{\binom{n}{i+1}} G_{i+2}$ lie in prefiltrations of degree at most $i+1$ and are therefore equal for all $n$ by the above lemma, allowing us to move on to the next stage of the construction. We are done once we have $g_0, \ldots, g_s$, since $G_{s+1}$ is trivial.
\end{proof}

\section{On the pairwise-independence condition}\label{appendix:pairwise-independence}

In this section we examine to what extent the pairwise-independence condition on the linear forms is needed for our main convergence results. First we note that these results do hold for systems of two linearly dependent forms.

\begin{lemma}
Let $\Lf$ consist of two integer linear forms $\lf_1,\lf_2$ with $\lf_2=k \lf_1$, for some integer $k\not\in\{0,1\}$.
Then as $p\to\infty$ through the primes we have $d_\Lf(\Zmod{p})\to 1/2$ and $m_\Lf(\alpha,p)\to m_\Lf(\alpha)$, where $m_\Lf(\alpha)$ equals $0$ for $\alpha\leq 1/2$ and $2\alpha-1$ for $\alpha>1/2$.
\end{lemma}

\begin{proof}
Let us start with $d_\Lf$. For $p$ large, it is easy to see that $A$ is $\Lf$-free if and only if $A \cap (k\cdot A) = \emptyset$, and we can construct such a set of density asymptotically $1/2$ relatively easily. Indeed, if $k=-1$ simply let $A = [(p-1)/2]$. Otherwise, let $n$ be the order of $k$ in the multiplicative group $\Zmod{p}^\times$. Let $H$ be the multiplicative subgroup $\{k^j:j\in [n]\}$, with $\Zmod{p}^\times=\sqcup_{j\in [m]} y_j\cdot H$, where $m=(p-1)/n=O_k(p/\log p)$. Let $E = \{k^{2j} : j \in [\floor{n/2}]\} \subset H$, and define $A= \sqcup_{j\in [m]} y_j\cdot E$. We have $A \cap (k\cdot A)= \emptyset$ and $\abs{A}/p=1/2+O_k(1/\log p)$. On the other hand, clearly $A\cap (k\cdot A)\neq\emptyset$ for any set $A$ of size at least $(p+1)/2$. Hence $d_\Lf(\Zmod{p})\to 1/2$.

Regarding $m_\Lf(\alpha,p)$, note first that for $\alpha<1/2$ any subset of density $\alpha$ of the set $A$ constructed above shows that $m_\Lf(\alpha,p)=0$. For $\alpha>1/2$, note the relationship $\Sol_\Lf(A')= 1 - 2\alpha + \Sol_\Lf(A)$ between a set $A \subset \Zmod{p}$ and its complement $A'$, as follows from the bilinearity of $\Sol_\Lf$. Since the $S_\Lf(A')$ term is always non-negative, letting $A' \subset \Zmod{p}$ be a $\Lf$-free set of density $1-\alpha$ then gives $A$ such that $\Sol_\Lf(A)=2\alpha-1=m_\Lf(\alpha,p)$.
\end{proof}

This proof provides a completely explicit extremal set $A$. As such, it is not obvious how to extend this result to systems of more than two forms, two of which are linearly dependent, or indeed to finite families of systems, one of which consists just of two dependent forms. To prove convergence in this setup, it seems that one would instead want a transference result for systems of two dependent forms, that would be compatible with the transference results we already have for systems of finite complexity. We shall now show that such a hypothetical transference result cannot be based on the uniformity norms---at least not in the usual way. Indeed we shall construct, given a system $\Lf$ of two dependent forms and any $d>1$, a family of sets in $\Zmod{p}$ for which the solution measure $\Sol_\Lf$ is not controlled by the $U^d$ norm.\footnote{Such a result is somewhat folklore, but this seems a suitable place to record it.} The sets we shall consider are essentially so-called \emph{Nil$_d$ Bohr sets} in $\Zmod{p}$ (see \cite{HKBohr}) and were already used to similar effects in \cite{GSz}.

\begin{proposition}
Let $\Lf$ consist of two integer linear forms $\lf_1,\lf_2$ with $\lf_2=k \lf_1$ for an integer $k$ with $\abs{k} \geq 2$. Let $d>1$ and set $\delta=1/4k^{2d}$. Then for any prime $p$ there is a set $A \subset \Zmod{p}$ such that 

\noindent \textup{(}i\textup{)}   $\Sol_\Lf (A)=0$,\hfill \textup{(}ii\textup{)}   $\alpha:= \abs{A}/p= 2\delta + o_{p\to\infty;k,d}(1)$, and \hfill \textup{(}iii\textup{)}   $\norm{ 1_A - \alpha }_{U^d}=o_{p\to\infty;k,d}(1)$.
\end{proposition}
\begin{proof}
We shall assume that $p$ is large, and for notational convenience we restrict to positive $k$. Let $I$ denote the interval $\left[\floor{p/k^d}-\delta p, \floor{p/k^d}+\delta p\right] $ in $\Zmod{p}$, and set
\[ A=\left\{x\in\Zmod{p}: x^d\in I \mod p\right\}. \]
Note that $\Sol_\Lf(A)= \abs{A\cap (k\cdot A)} / p$, and that if $y=kx\in k\cdot A$ then $y^d\in k^d\cdot I\subset [p-p/2k^d,p+p/2k^d]$. But the latter interval is disjoint from $I$ since $p/2k^d< \floor{p/k^d}-\delta p$, hence the first property of the conclusion holds.

To establish the other two properties we shall use the Fourier transform, defined as $\widehat{f}(r) = \E_{x \in \Zmod{p}} f(x) e(-r\cdot x)$.\footnote{Here as usual $r\cdot x=rx/p$, and $e(\theta)=\exp(2\pi i \theta)$ for any $\theta\in \T$.} First note that by Fourier inversion we have
\begin{equation}\label{inversion}
	1_A(x) = 1_I(x^d) = \sum_t \widehat{1_I}(t)e(t \cdot x^d).
\end{equation}
We shall use this expression together with the two standard estimates $\sum_t \abs{\widehat{1_I}(t)} = O(\log p)$ and $\norm{ e(t \cdot x^d)}_{U^d} = O(p^{-1/2^d})$ for $t$ non-zero, the latter being essentially a Weyl differencing estimate; see e.g. \cite[Exercise 11.1.12]{T-V}. For (ii), then, we have
\[ \alpha = \E_x 1_A(x) = \abs{I}/p + \sum_{t \neq 0} \widehat{1_I}(t) \E_x e(t \cdot x^d) \]
and the latter expression is at most $\sum_{t \neq 0} \abs{\widehat{1_I}(t)} \cdot \abs{\E_x e(t \cdot x^d)} = O(p^{-1/2^d} \log p)$. For (iii), coupling \eqref{inversion} with the $U^d$ triangle inequality yields
\[ \norm{1_A - \alpha}_{U^d} \leq \sum_{t \neq 0} \abs{\widehat{1_I}(t)} \norm{ e(t\cdot x^d) - \E_y e(t \cdot y^d) }_{U^d} = O(p^{-1/2^d} \log p), \]
and we are done.
\end{proof}
The set $A$ given by this proposition is thus virtually indistinguishable from the constant function $\alpha$ from the point of view of the $U^d$ uniformity norm, whereas $\Sol_\Lf(A)$ and $\Sol_\Lf(\alpha)$ are not at all close.

\section{The arithmetic removal lemma}\label{appendix:removal}

In this last appendix we establish Theorem \ref{removal}, which we restate here.

\begin{theorem}
Let $\Lf$ be a system of linear forms $\lf_1,\ldots,\lf_t:\Z^D\to \Z$. Then there is a positive integer $K$ such that the following holds. For any $\epsilon>0$, there exists $\delta=\delta(\epsilon,\Lf)>0$ such that if $N\in \N$ is prime to $K$, and $A_1,\ldots, A_t$ are subsets of $\Zmod{N}$ such that $\Sol_\Lf(A_1,\ldots,A_t)\leq \delta$, then there exist sets $E_i\subset \Zmod{N}$ with $\abs{E_i}\leq \epsilon N$ for all $i\in [t]$, such that $\Sol_\Lf(A_1\setminus E_1,\ldots,A_t\setminus E_t)=0$.
\end{theorem}
\begin{proof}
This will follow from the removal lemma \cite[Theorem 1]{KSVGR} of Kr\'al', Serra and Vena provided we can find a homomorphism (or matrix) $\Lambda : \Z^t \to \Z^k$ such that $\ker_{\Zmod{N}} \Lambda = \Lf(\Zmod{N}^D)$, since then $\Sol_\Lf(A_1,\ldots, A_t)=\abs{A_1\times \cdots \times A_t\, \cap\, \ker_{\Zmod{N}} \Lambda}/\abs{\ker_{\Zmod{N}} \Lambda}$.
Here
\begin{align*}
	\ker_{\Z_N} \Lambda &= \{ y + N \Z^t \in \Z_N^t : \Lambda(y) \in N \Z^k \}, \\
		\Lf(\Z_N^D) &= \{ \Lf(x) + N\Z^t : x \in \Z^D \}.
\end{align*}
We construct such a $\Lambda$ in stages. First let $f : \Z^t \to \Z^t/\Lf(\Z^D)$ be the quotient map $x \mapsto x + \Lf(\Z^D)$. The target of this map, being finitely generated, is isomorphic to $Z := \Z^k \times \Zmod{N_1} \times \cdots \times \Zmod{N_r}$ for some integers $k \geq 0$ and $N_j \in \N$; let $g$ be a corresponding isomorphism. Assume $N$ is prime to each $N_j$; then $N \cdot Z = (N\Z^k) \times \Zmod{N_1} \times \cdots \times \Zmod{N_r}$. We claim that $\Lambda := \pi \circ g \circ f$ satisfies the required relationship, where $\pi$ denotes projection to $\Z^k$. Indeed, writing $A \oplus H = \{ a + H : a \in A \} \subset G/H$ for a set $A \subset G$ and a subgroup $H \leq G$, we have
\[ \Lambda^{-1}(N\Z^k) = f^{-1}(g^{-1}(N Z)) = f^{-1}(N\Z^t \oplus \Lf(\Z^D)) = N\Z^t + \Lf(\Z^D), \]
the second equality following from $g$ being an isomorphism. Reducing mod $N \Z^t$ gives the required relationship.
\end{proof}

\end{document}